\documentclass[11pt]{article}
\usepackage[margin=1in]{geometry} 
\usepackage{amssymb,amsfonts,amsmath,bbm,mathrsfs,stmaryrd}
\usepackage{xcolor}
\usepackage{url}

\usepackage{enumerate}

\usepackage{hyperref}
\hypersetup{colorlinks,
             linkcolor=black!75!red,
             citecolor=blue,
             pdftitle={},
             pdfproducer={pdfLaTeX},
             pdfpagemode=None,
             bookmarksopen=true
             bookmarksnumbered=true}

\usepackage{tikz}
\usetikzlibrary{arrows,calc,decorations.pathreplacing,decorations.markings,intersections,shapes.geometric,through,fit,shapes.symbols,positioning,decorations.pathmorphing}

\usepackage{braket}

\usepackage[amsmath,thmmarks,hyperref]{ntheorem}
\usepackage{cleveref}

\creflabelformat{enumi}{#2(#1)#3}

\crefname{section}{Section}{Sections}
\crefformat{section}{#2Section~#1#3} 
\Crefformat{section}{#2Section~#1#3} 

\crefname{subsection}{\S}{\S\S}
\AtBeginDocument{%
  \crefformat{subsection}{#2\S#1#3}%
  \Crefformat{subsection}{#2\S#1#3}%
}

\crefname{subsubsection}{\S}{\S\S}
\AtBeginDocument{%
  \crefformat{subsubsection}{#2\S#1#3}%
  \Crefformat{subsubsection}{#2\S#1#3}%
}

\theoremstyle{plain}

\newtheorem{lemma}{Lemma}[section]
\newtheorem{proposition}[lemma]{Proposition}
\newtheorem{corollary}[lemma]{Corollary}
\newtheorem{theorem}[lemma]{Theorem}

\theoremstyle{nonumberplain}
\newtheorem{theoremN}{Theorem}

\theoremstyle{plain}
\theorembodyfont{\upshape}
\theoremsymbol{\ensuremath{\blacklozenge}}

\newtheorem{definition}[lemma]{Definition}
\newtheorem{example}[lemma]{Example}
\newtheorem{remark}[lemma]{Remark}

\crefname{definition}{definition}{definitions}
\crefformat{definition}{#2definition~#1#3} 
\Crefformat{definition}{#2Definition~#1#3} 

\crefname{ex}{example}{examples}
\crefformat{example}{#2example~#1#3} 
\Crefformat{example}{#2Example~#1#3} 

\crefname{remark}{remark}{remarks}
\crefformat{remark}{#2remark~#1#3} 
\Crefformat{remark}{#2Remark~#1#3} 

\crefname{convention}{convention}{conventions}
\crefformat{convention}{#2convention~#1#3} 
\Crefformat{convention}{#2Convention~#1#3} 

\crefname{notation}{notation}{notations}
\crefformat{notation}{#2notation~#1#3} 
\Crefformat{notation}{#2Notation~#1#3} 

\crefname{table}{table}{tables}
\crefformat{table}{#2table~#1#3} 
\Crefformat{table}{#2Table~#1#3}

\crefname{lemma}{lemma}{lemmas}
\crefformat{lemma}{#2lemma~#1#3} 
\Crefformat{lemma}{#2Lemma~#1#3} 

\crefname{proposition}{proposition}{propositions}
\crefformat{proposition}{#2proposition~#1#3} 
\Crefformat{proposition}{#2Proposition~#1#3} 

\crefname{corollary}{corollary}{corollaries}
\crefformat{corollary}{#2corollary~#1#3} 
\Crefformat{corollary}{#2Corollary~#1#3} 

\crefname{theorem}{theorem}{theorems}
\crefformat{theorem}{#2theorem~#1#3} 
\Crefformat{theorem}{#2Theorem~#1#3} 

\crefname{enumi}{}{}
\crefformat{enumi}{(#2#1#3)}
\Crefformat{enumi}{(#2#1#3)}

\crefname{assumption}{assumption}{Assumptions}
\crefformat{assumption}{#2assumption~#1#3} 
\Crefformat{assumption}{#2Assumption~#1#3} 

\crefname{equation}{}{}
\crefformat{equation}{(#2#1#3)} 
\Crefformat{equation}{(#2#1#3)}

\numberwithin{equation}{section}
\renewcommand{\theequation}{\thesection-\arabic{equation}}

\theoremstyle{nonumberplain}
\theoremsymbol{\ensuremath{\blacksquare}}

\newtheorem{proof}{Proof}
\newcommand\pf[1]{\newtheorem{#1}{Proof of \Cref{#1}}}

\newcommand\bC{{\mathbb C}}

\newcommand\bG{{\mathbb G}}

\newcommand\bN{{\mathbb N}}

\newcommand\bS{{\mathbb S}}

\newcommand\bX{{\mathbb X}}
\newcommand\bY{{\mathbb Y}}
\newcommand\bZ{{\mathbb Z}}

\newcommand\cC{{\mathcal C}}
\newcommand\cD{{\mathcal D}}
\newcommand\cE{{\mathcal E}}

\newcommand\cM{{\mathcal M}}

\newcommand\cP{{\mathcal P}}
\newcommand\cQ{{\mathcal Q}}

\newcommand\cU{{\mathcal U}}

\DeclareMathOperator{\id}{id}

\newcommand\numberthis{\addtocounter{equation}{1}\tag{\theequation}}

\newcommand{\cat}[1]{\textsc{#1}}

\newcommand{\qedhere}{\mbox{}\hfill\ensuremath{\blacksquare}}

\title{Non-commutative ambits and equivariant compactifications}
\author{Alexandru Chirvasitu}

\begin{document}

\date{}

\newcommand{\Addresses}{{
  \bigskip
  \footnotesize

  \textsc{Department of Mathematics, University at Buffalo, Buffalo,
    NY 14260-2900, USA}\par\nopagebreak \textit{E-mail address}:
  \texttt{achirvas@buffalo.edu}

}}

\maketitle

\begin{abstract}
  We prove that an action $\rho:A\to M(C_0(\mathbb{G})\otimes A)$ of a locally compact quantum group on a $C^*$-algebra has a universal equivariant compactification, and prove a number of other category-theoretic results on $\mathbb{G}$-equivariant compactifications: that the categories compactifications of $\rho$ and $A$ respectively are locally presentable (hence complete and cocomplete), that the forgetful functor between them is a colimit-creating left adjoint, and that epimorphisms therein are surjective and injections are regular monomorphisms.

  When $\mathbb{G}$ is regular coamenable we also show that the forgetful functor from unital $\mathbb{G}$-$C^*$-algebras to unital $C^*$-algebras creates finite limits and is comonadic, and that the monomorphisms in the former category are injective.
\end{abstract}


\vspace{.5cm}


\tableofcontents

\section*{Introduction}

Consider a continuous action
\begin{equation*}
  \pi:\bG\times\bX\to \bX
\end{equation*}
of a topological group on a topological space. While there is always a universal, functorial map $\bX\to \beta\bX$ into a compact Hausdorff space (the {\it Stone-\v{C}ech compactification} of $\bX$ \cite[\S 38]{mnk}), it is only very rarely that the induced action $\bG\times\beta\bX\to \beta\bX$ is continuous. The original motivation for the present paper was the fact that nevertheless, there is always a universal {\it $\bG$-equivariant compactification} of $\bX$. In short (and more precisely), the inclusion functor
\begin{equation*}
  \cat{CH}^{\bG}\to \cat{Top}^{\bG}
\end{equation*}
from compact Hausdorff $\bG$-spaces into topological $\bG$-spaces has a left adjoint; see \cite[\S 2.8]{dvr-puc} or the much more extensive discussion in \cite[\S 4.3]{dvr-bk}. That left adjoint then specializes back to $\bX\mapsto\beta\bX$ when $\bG$ is trivial. The (greatest) {\it ambit} of $\bG$ (\cite[Introduction]{brk}), featuring in the title of this paper, is the result of applying that left adjoint to the translation self-action $\bG\times\bG\to \bG$.

The aim here is to examine the same types of questions in the context of {\it non-commutative} topology (i.e. $C^*$-algebras). Topological groups, in particular, are now locally compact {\it quantum} groups (LCQGs) in the sense of \cite[Definition 4.1]{kvcast}:

\begin{definition}\label{def:lcqg}
  A {\it locally compact quantum group} $\bG$ consists of a $C^*$-algebra $C=C_0(\bG)$ equipped with a non-degenerate morphism
  \begin{equation*}
    \Delta=\Delta_{\bG}:C\to M(C\otimes C)
  \end{equation*}
  such that
  \begin{itemize}
  \item we have $(\Delta\otimes\id)\Delta=(\id\otimes\Delta)\Delta$ as maps $C\to M(C^{\otimes 3})$;    
  \item the sets
    \begin{equation*}
      \{(\omega\otimes\id)\Delta(a)\ |\ \omega\in C^*,\ a\in A\}
      \quad\text{and}\quad
      \{(\id\otimes\omega)\Delta(a)\ |\ \omega\in C^*,\ a\in A\}
    \end{equation*}
    are contained in $A$ and span dense subsets therein;
  \item and $C$ is equipped with left and right-invariant approximate KMS weights. 
  \end{itemize}
\end{definition}
The weights in last item are analogues of the left and right Haar measures on a locally compact group (hence the name: {\it Haar weights}); we will have no need to elaborate on the precise meaning of that part of the definition, as it plays only an indirect role in the discussion below.

Apart from whatever intrinsic interest generalization for its own sake might hold, recasting topological results in $C^*$-algebraic terms occasionally has the effect of sharpening the proofs by divorcing them from some of the unnecessary point-set topological baggage. Additionally, this type of generalization can also raise its own peculiar problems that might have been difficult to notice or appreciate in the ``classical'' setup of ordinary spaces/actions. 

With this in mind, the dictionary is as follows:
\begin{itemize}
\item in place of an ordinary space $\bX$ we have a possibly non-commutative $C^*$-algebra $A=C_0(\bX)$ (of ``continuous functions vanishing at infinity'' on a non-commutative space);
\item $\bG$ is a locally compact quantum group as above, reified as its own function algebra $C_0(\bG)$;
\item $\pi$ becomes instead an {\it action}
  \begin{equation}\label{eq:rhointro}
    \rho:A\to M(C_0(\bG)\otimes A),
  \end{equation}
  in the sense of \Cref{def:act} below, where $M(-)$ denotes the {\it multiplier algebra} construction \cite[\S II.7.3]{blk};
\item and (perhaps as expected from perusing the classical machinery), a {\it $\bG$-equivariant compactification} of $\rho$ (\Cref{def:cpct,def:cats}) is a $\bG$-action on a {\it unital} $C^*$-algebra $B$ ``sandwiched'' equivariantly between $A$ and its multiplier algebra $M(A)$ (the latter being the non-commutative analogue of $\beta\bX$).
\end{itemize}

One of the results of this paper, then, is that equivariant compactifications exist in the non-commutative setup sketched above (\Cref{cor:term}):

\begin{theoremN}
  Every action by a locally compact quantum group on a $C^*$-algebra admits a universal equivariant compactification. \qedhere
\end{theoremN}

The proof is essentially category-theoretic, and suggests its own separate line of investigation into the nature of various categories of quantum actions or spaces. Specifically, given an action \Cref{eq:rhointro}, we focus in \Cref{se:cats} on the category
\begin{itemize}
\item $\cat{cpct}(\rho)$ of equivariant compactifications of $A$;
\item $\cat{cpct}(A)$ of compactifications of (the noncommutative space underlying) $A$;
\item and $A\downarrow \cC^*_1$, of $C^*$ morphisms from $A$ into unital $C^*$-algebras.
\end{itemize}
These are progressively poorer in structure, with forgetful functors pointing down the chain as listed here (see \Cref{def:cpct,def:cats} and subsequent discussion). A brief summary of some of the results (\Cref{th:ccmpl,th:episurj,th:injeq,th:allpres} and \Cref{cor:allgood,cor:ladj}) reads:

\begin{theoremN}
  Let \Cref{eq:rhointro} be an action of an LCQG on a $C^*$-algebra.
  \begin{enumerate}[(a)]
  \item The categories $\cat{cpct}(\rho)$, $\cat{cpct}(A)$ and $A\downarrow\cC^*_1$ are all locally presentable, hence also complete and cocomplete.
  \item The forgetful functors
    \begin{equation*}
      \cat{cpct}(\rho)
      \to
      \cat{cpct}(A)
      \to
      A\downarrow\cC^*_1
    \end{equation*}
    are all left adjoints and create colimits.
  \item In any of these categories, the epimorphisms are the surjections.
  \item And injective morphisms are equalizers (i.e. regular monomorphisms).
    \qedhere
  \end{enumerate}
\end{theoremN}
Although on a first encounter {\it local presentability} (\Cref{def:lpres}) might appear a little technical, it reflects fairly reasonable intuition (that the category can be recovered via colimits from ``small'' objects) and is extremely valuable in practice in proving the kinds of universality results we are concerned with here: as the proofs of \Cref{cor:allgood,cor:ladj} make it clear for instance, local presentability delivers limits and adjoint functors essentially ``for free''.

In \Cref{se:lim} we focus on the forgetful functor $\cC^{*\bG}_1\to \cC^*_1$ from actions \Cref{eq:rhointro} on unital $C^*$-algebras to (again unital) $C^*$-algebras; it is recoverable from the previous discussion as 
\begin{equation*}
  \cat{cpct}(\rho)\to \cat{cpct}(A),
\end{equation*}
where $A$ is the zero algebra $\{0\}$. Much of \Cref{se:lim} specializes to {\it coamenable} $\bG$ (\cite[Definition 3.1]{bt} and \Cref{def:coam}). On the one hand, this is a non-trivial constraint; on another though, it is also an instance of the phenomenon noted above, of encountering purely non-commutative phenomena when generalizing results from classical point-set topology: ordinary locally compact groups are all coamenable, so the issue was invisible in the classical setting.

The main results there are (\Cref{th:finlim,th:coamcomon} and \Cref{cor:moninj}):

\begin{theoremN}
  Let $\bG$ be a regular coamenable locally compact quantum group.
  \begin{enumerate}[(a)]
  \item The forgetful functor $\cC^{*\bG}_1\to \cC^*_1$ creates finite limits and is comonadic.
  \item And the monomorphisms in $\cC^{*\bG}_1$ are the injections.  \qedhere
  \end{enumerate}
\end{theoremN}

As explained in \Cref{subse:cmnd}, the comonadicity claim formalizes the intuition that a non-commutative $\bG$-space is precisely a non-commutative space equipped with additional ``algebraic operations'', and that the forgetful functor forgets that additional algebraic structure.

\subsection*{Acknowledgements}

This work was partially supported through NSF grant DMS-2001128. 

I am grateful for very constructive comments from A. Viselter. 

\section{Preliminaries}\label{se.prel}

We make casual use of operator-algebraic background, as available, say, in \cite{blk,ped-aut,tak1} and numerous other good sources; the references are more precise where appropriate. References for locally compact quantum groups include \cite{kvcast,bs-cross,mrw}, again with more specific citations in the text below.

The unadorned tensor-product symbol `$\otimes$', placed between $C^*$-algebras, denotes the {\it minimal} or {\it spatial} $C^*$ tensor product \cite[\S II.9.1.3]{blk}. We will, on occasion, also refer to the {\it maximal} $C^*$ tensor product $\overline{\otimes}$ \cite[\S II.9.2]{blk}.

\subsection{Multipliers and categories of $C^*$-algebras}\label{subse:catcast}

A morphism between unital $C^*$-algebras can only, reasonably, mean one thing: a continuous, unital, $*$-preserving, multiplicative and linear map. We denote the resulting category by $\cC^*_1$, to remind the reader of the unitality.

{\it Non}-unital $C^*$-algebras form categories in a number of competing ways (all collapsing back to $\cC^*_1$ in the unital case):
\begin{enumerate}[(a)]
\item\label{item:1} $\cC^*$, where morphisms are as in $\cC^*_1$, simply dropping the unitality requirement (which no longer makes sense). This category is not frequently useful as-is, so there are variants to consider.
\item\label{item:2} $\underline{\cC^*}$, where the morphisms are those of $\cC^*$ which are additionally {\it proper} \cite[\S 2.1]{elp}: $f:A\to B$ is proper if
  \begin{equation}\label{eq:ndg}
    \overline{f(A)B} = B = \overline{B f(A)}. 
  \end{equation}
  As explained on \cite[p.80]{elp}, in the classical case (i.e. for commutative $C^*$-algebras) these dualize precisely to the proper continuous maps between locally compact Hausdorff spaces.
\item\label{item:3} $\overline{\cC^*}$ (the most useful version in quantum-group theory, it seems), where a morphism $A\to B$ is by definition a {\it non-degenerate} morphism $A\to M(B)$ into the multiplier algebra of $B$, where recall \cite[p.15]{lnc-hilb} that non-degeneracy is here the selfsame condition \Cref{eq:ndg}.

  The set of such morphisms $A\to B$ in $\overline{\cC^*}$ is often denoted by $\mathrm{Mor}(A,B)$: \cite[\S 0]{wor-aff}, \cite[Definition A.4]{mnw}, \cite[\S 1.1]{dkss}, etc. Such a morphism extends uniquely \cite[Proposition 2.5]{lnc-hilb} to a unital $C^*$-morphism $M(A)\to M(B)$ that is furthermore strictly continuous on the unit ball of $M(A)$ (or equivalently \cite[Corollary 2.7]{taylor}, just plain strictly continuous). This then allows for composition of morphisms.

  We will usually denote extensions $M(A)\to M(B)$ of non-degenerate $f:A\to M(B)$ by the same symbol, relying on context for telling the two apart.
\end{enumerate}

\begin{remark}\label{re:cats}
  Note that the proper morphisms $A\to B$ are precisely the non-degenerate morphisms $A\to M(B)$ that actually take values in $B$. This means that we have an inclusion
  \begin{equation*}
    \underline{\cC^*} \subset \overline{\cC^*}
  \end{equation*}
  of categories. Neither, though, is comparable with $\cC^*$.
\end{remark}

It will also pay off to introduce notation that singles out certain subalgebras of multiplier algebras. In \cite[paragraph preceding D\'efinition 0.1]{bs-cross} the authors of that paper write
\begin{equation*}
  M(A;\ J) := \{x\in M(A)\ |\ xJ+Jx\subseteq J\}
\end{equation*}
for an ideal $J\trianglelefteq A$. As noted there, in addition to being a $C^*$-subalgebra of $M(A)$ by definition, $M(A;\ J)$ can also be regarded as a $C^*$-subalgebra of $M(J)$ via the multiplier-restriction morphism
\begin{equation*}
  M(A;\ J)\subseteq M(A)\to M(J).
\end{equation*}
The claim here is that this composition is one-to-one; we leave that as an exercise, and henceforth take it for granted.

Of particular interest, in defining actions \cite[D\'efinition 0.2]{bs-cross}, will be algebras of the form
\begin{equation*}
  M(B\otimes A^+;\ B\otimes A) = M(B\otimes \widetilde{A};\ B\otimes A),
\end{equation*}
where $A^+$ and $\widetilde{A}$ are as in \cite[Propositions 2.1.3 and 2.1.7]{wo}:
\begin{itemize}
\item $A^+=A\oplus \bC$ with $A$ as an ideal and $(0,1)$ as the unit;
\item $\widetilde{A}$ is the {\it smallest unitization} of $A$, i.e. $A$ itself if the latter was already unital and $A^+$ otherwise. 
\end{itemize}

There is a tradition (e.g. \cite[paragraph following Lemma 1.1]{lprs} or \cite[\S 1]{qg-full}) of writing 
\begin{equation*}
  \widetilde{M}(A\otimes B) = M(\widetilde{A}\otimes B;\ A\otimes B).  
\end{equation*}
Since this paper's conventions demand that the affected tensorand $A$ be placed on the {\it right} (rather than left), we decorate `$M$'with that tensorand as a subscript instead, to the notation. Thus:
\begin{align*}
  M_A(A\otimes B) &:= M(\widetilde{A}\otimes B;\ A\otimes B) = M(A^+\otimes B;\ A\otimes B) ,\\
  M_A(B\otimes A) &:= M(B\otimes \widetilde{A};\ B\otimes A) = M(B\otimes A^+;\ B\otimes A),\numberthis\label{eq:msub}\\
\end{align*}
etc.

\begin{remark}\label{re:when1}
  Naturally, when $A$ is unital $M_A(B\otimes A)$ specializes back to the usual multiplier algebra $M(B\otimes A)$.
\end{remark}

Although, as mentioned above, the $C^*$ categories $\underline{\cC^*}$ and $\overline{\cC^*}$ tend to be more useful in practice than $\cC^*$, the latter does feature in the following observation.

\begin{lemma}\label{le:isfnc}
  For a fixed $C^*$-algebra $C$, the construction
  \begin{equation}\label{eq:isfnc}
    A\mapsto M_A(C\otimes A)
  \end{equation}
  is an endofunctor on $\cC^*$.

  Furthermore, that functor preserves morphism injectivity.
\end{lemma}
\begin{proof}
  We just focus on how the functor applies to morphisms; the fact that it respects compositions and identities will then be routine.

  Consider a $C^*$-morphism $f:A\to B$ (in $\cC^*$; no non-degeneracy assumptions, etc.). It then extends \cite[\S II.1.2.3]{blk} uniquely to a unital morphism
  \begin{equation*}
    \id\otimes f: C\otimes A^+\to C\otimes \widetilde{B}
  \end{equation*}
  which is proper in the sense of \Cref{item:2} in the above discussion on categories. This means that it extends strictly continuously to a unital morphism
  \begin{equation}\label{eq:isfnc+}
    \id\otimes f: M(C\otimes A^+)\to M(C\otimes \widetilde{B}),
  \end{equation}
  and finally to \Cref{eq:isfnc} because this last extension has the requisite multiplier property: for $x\in M_A(C\otimes A^+)$ and $c\in C$ we have
  \begin{align*}
    (\id\otimes f)(x)(c\otimes 1) &= (\id\otimes f)(x(c\otimes 1)) \\
                                  &\in (\id\otimes f)(C\otimes A)
                                    \quad\text{because $x\in M_A(C\otimes A^+)$}\\
                                  &\subseteq C\otimes B.\\
  \end{align*}
  As this happens for arbitrary $c\in C$, \Cref{eq:isfnc+} indeed restricts to a morphism
  \begin{equation*}
    M_A(C\otimes A)\to M_B(C\otimes B).
  \end{equation*}
  As for the injectivity-preservation claim, recall first \cite[Proposition IV.4.22]{tak1} that
  \begin{equation*}
    f:A\to B\text{ injective }\Rightarrow\ \id\otimes f:C\otimes A\to C\otimes B\text{ is injective}.
  \end{equation*}
  The conclusion then follows from the fact that the latter morphism extends to $M_A(C\otimes A)$, wherein $C\otimes A$ is an {\it essential} ideal: the kernel of that extension, if non-trivial, would intersect $C\otimes A$ non-trivially.
\end{proof}

\begin{remark}
  It is observed in \cite[Lemma 1.4]{lprs} that for $C^*$-algebras $A$ and $C$ and an ideal $I\trianglelefteq A$ the restriction morphism
  \begin{equation*}
    M(C\otimes A)\to M(C\otimes I)
  \end{equation*}
  identifies the ideal
  \begin{equation*}
    \{m\in M_A(C\otimes A)\ |\ mx,\ xm\in C\otimes I,\ \forall x\in C\otimes A\}\subseteq M_A(C\otimes A)
  \end{equation*}
  with $M_I(C\otimes I)\subseteq M(C\otimes I)$. That identification (or rather its inverse) is easily seen to be precisely the embedding
  \begin{equation*}
    M_I(C\otimes I)\to M_A(C\otimes A)
  \end{equation*}
  provided by \Cref{le:isfnc}.
\end{remark}

A few further remarks on multiplier algebras follow, for occasional future reference.

Recall \cite[\S 1.4]{wo} that an {\it essential} ideal in a $C^*$-algebra is one having non-zero intersection with every non-zero ideal. A well-known result, recorded here for future reference.

\begin{lemma}\label{le:ess}
  If $I\trianglelefteq A$ and $J\trianglelefteq B$ are essential $C^*$-ideals then so is $I\otimes J\trianglelefteq A\otimes B$.

  Equivalently, the morphism
  \begin{equation}\label{eq:abmij}
    A\otimes B\to M(I\otimes J)
  \end{equation}
  obtained by realizing $A\otimes B$ as an algebra of multipliers on $I\otimes J$ is one-to-one.
\end{lemma}
\begin{proof}
  That the two claims re equivalent is part of \cite[Proposition 2.2.14]{wo}. \cite[Proposition 5.1]{skf} proves the statement for $A=M(I)$ and $B=M(J)$. $A\to M(I)$ and $B\to M(J)$ are embeddings by the same \cite[Proposition 2.2.14]{wo} and minimal $C^*$ tensor products preserve map injectivity \cite[Proposition IV.4.22]{tak1}, so the injectivity of \Cref{eq:abmij} follows from writing as a composition of two injective maps
  \begin{equation*}
    A\otimes B \subseteq M(I)\otimes M(J)\subseteq M(I\otimes J). 
  \end{equation*}
\end{proof}

In particular, applying \Cref{le:ess} to the essential ideals $A\trianglelefteq M(A)$ and $B\trianglelefteq B$ we obtain

\begin{corollary}\label{cor:mmab}
  For $C^*$-algebras $A$ and $B$ the canonical morphism
  \begin{equation}\label{eq:mabab}
    M(M(A)\otimes B)\to M(A\otimes B)
  \end{equation}
  resulting from realizing $M(A)\otimes B$ as an algebra of multipliers on $A\otimes B$ is one-to-one.  \qedhere
\end{corollary}

\begin{remark}\label{re:mma}
  The discussion in \cite[\S 3.1, second paragraph]{vrgn-phd} traces the general failure of an action $A\circlearrowleft \bG$ to extend to $M(A)\circlearrowleft \bG$ to the fact that in general, \Cref{eq:mabab} is a {\it proper} embedding. Recall the distinction between the two sides in the ``half-classical'' case when $B=C_0(\bX)$ for some locally compact space $\bX$:
  \begin{itemize}
  \item $M(A\otimes B)=M(C_0(\bX,A))$ is the algebra of bounded, strictly continuous functions $\bX\to M(A)$ \cite[Corollary 3.4]{apt};
  \item whereas by the same result, $M(M(A)\otimes B)=M(C_0(\bX,M(A)))$ is the algebra of bounded {\it norm}-continuous functions $\bX\to M(A)$.
  \end{itemize}
  Since the latter, norm-continuity requirement is generally stronger than strict continuity, \Cref{eq:mabab} will indeed generally be a proper embedding.
\end{remark}

\subsection{Actions}\label{subse:act}

\begin{definition}\label{def:act}
  Let $A$ be a $C^*$-algebra and $\bG$ a locally compact quantum group. An {\it action} $A \circlearrowleft \bG$ is a non-degenerate morphism $\rho:A\to M(C_0(\bG)\otimes A)$ such that
  \begin{enumerate}[(a)]
  \item\label{item:6} $(\id\otimes\rho)\rho = (\Delta_{\bG}\otimes\id)\rho$ as morphisms $A\to M(C_0(\bG)^{\otimes 2}\otimes A)$;
  \item\label{item:7} $\rho$ takes values in $M_A(C_0(\bG)\otimes A)$ (notation as in \Cref{eq:msub}), so that
    \begin{equation*}
      \rho(A)(C_0(\bG)\otimes \bC)\subseteq C_0(\bG)\otimes A;
    \end{equation*}
  \item\label{item:8} we have
    \begin{equation}\label{eq:aactcont}
      \overline{\rho(A)(C_0(\bG)\otimes \bC)}^{\|\cdot \|} = C_0(\bG)\otimes A.
    \end{equation}
  \end{enumerate}
  We will occasionally also denote such a gadget by $\rho:A \circlearrowleft \bG$.
\end{definition}

\begin{remark}
  The notation $A\circlearrowleft \bG$ is meant to suggest the left-right switch that occurs every time we pass from a space to its associated function algebra, and from an action to a coaction. Thus:
  \begin{itemize}
  \item a {\it left} coaction $\rho:A\to M(C_0(\bG)\otimes A)$ as in \Cref{def:act};
  \item counts as a {\it right} action of $\bG$ on $A$;
  \item and hence a {\it left} action of $\bG$ on the ``noncommutative space'' dual to $A$.
  \end{itemize}
\end{remark}


\begin{remark}\label{re:ndegauto}
  The non-degeneracy requirement in \Cref{def:act} might be needed to make sense of condition \Cref{item:6} (so as to be able to define the relevant compositions, etc.), but assuming conditions \Cref{item:7} and \Cref{item:8}, non-degeneracy {\it follows}:
  \begin{align*}
    C_0(\bG)\otimes A &= C_0(\bG)\otimes A^2
                        \quad\text{by {\it Cohen factorization} \cite[Theorem (32.22)]{hr2}: $A^2=A$}\\
                      &\subseteq \overline{\rho(A)(C_0(\bG)\otimes \bC)(\bC\otimes A)}^{\|\cdot\|}
                        \quad\text{by \Cref{eq:aactcont}}\\
                      &\subseteq \overline{\rho(A)(C_0(\bG)\otimes A)}^{\|\cdot\|},
  \end{align*}
  as desired.
  
  For this reason, when checking candidate actions against \Cref{def:act}, we will not belabor degeneracy much.
\end{remark}

\begin{example}\label{ex:regred}
The comultiplication
\begin{equation*}
  \Delta_{\bG}:C_0(\bG)\to M(C_0(\bG)\otimes C_0(\bG))
\end{equation*}
is the {\it regular} action $C_0(\bG)\circlearrowleft \bG$. That it is indeed an action follows, for instance, from \cite[Proposition 5.1]{wor-mult} and \cite[Proposition 6.10]{kvcast}.
\end{example}

\begin{example}\label{ex:reguniv}
  Similarly, the {\it universal} function algebra $C_0^u(\bG)$ ($A_u$ in \cite[\S 5]{kus-univ}, where $A$ is our $C_0(\bG)$) also comes equipped with a $\bG$-action
  \begin{equation*}
    (\pi\otimes\id)\Delta^u_{\bG}:C_0^u(\bG)\to M(C_0(\bG)\otimes C_0^u(\bG)).
  \end{equation*}
  Here,
  \begin{equation*}
    \Delta_{\bG}^u: C_0^u(\bG)\to M(C_0^u(\bG)\otimes C_0^u(\bG))
  \end{equation*}
  is the universal version of the comultiplication \cite[Proposition 6.1]{kus-univ}, $\pi:C_0^u(\bG)\to C_0(\bG)$ is the canonical surjection \cite[Notation 5.1]{kus-univ}, and the fact that once more this is an action in the sense of \Cref{def:act} follows from \cite[Proposition 6.1]{kus-univ}.
\end{example}

We will also have to consider morphisms of $C^*$-algebras equipped with actions by some fixed $\bG$. The notion of equivariance with respect to $\bG$-actions makes sense for a broader class of morphisms than those that make up any of the three categories discussed above: arbitrary $C^*$-morphisms $A\to M(B)$ will do. Cf. \cite[\S 4.3]{vrgn-phd}, where the concept (of equivariance) is essentially the same: there too, possibly-degenerate morphisms $A\to M(B)$ are considered.

\begin{definition}\label{def:eqvr}
  Let $\rho_A:A\circlearrowleft \bG$ and $\rho_B\circlearrowleft\bG$ be two actions by an LCQG. A $C^*$-morphism $f:A\to M(B)$ is {\it ($\bG$-)equivariant} if the diagram
  \begin{equation*}
    \begin{tikzpicture}[auto,baseline=(current  bounding  box.center)]
      \path[anchor=base] 
      (0,0) node (l) {$A$}
      +(4,.5) node (u) {$M(B)$}
      +(2,-1.5) node (dl) {$M_A(C_0(\bG)\otimes A)$}
      +(7,-1.5) node (dr) {$M(C_0(\bG)\otimes M(B))$}
      +(9,0) node (r) {$M(C_0(\bG)\otimes B)$}
      ;
      \draw[->] (l) to[bend left=6] node[pos=.5,auto] {$\scriptstyle f$} (u);
      \draw[->] (u) to[bend left=6] node[pos=.5,auto] {$\scriptstyle \rho_B$} (r);
      \draw[->] (l) to[bend right=6] node[pos=.5,auto,swap] {$\scriptstyle \rho_A$} (dl);
      \draw[->] (dl) to[bend right=6] node[pos=.5,auto,swap] {$\scriptstyle \id\otimes f$} (dr);
      \draw[->] (dr) to[bend right=6] node[pos=.5,auto,swap] {$\scriptstyle \subseteq $} (r);
    \end{tikzpicture}
  \end{equation*}
  commutes, where
  \begin{itemize}
  \item the bottom arrow $\id\otimes f$ is that provided by the functor of \Cref{eq:isfnc}, having identified
    \begin{equation*}
      M(C_0(\bG)\otimes M(B)) = M_{M(B)}(C_0(\bG)\otimes M(B))
    \end{equation*}
    (because $M(B)$ is unital);
  \item the bottom right-hand inclusion is that of \Cref{cor:mmab}. 
  \end{itemize}
\end{definition}

\section{Categories of noncommutative structured spaces}\label{se:cats}

As noted in \cite[\S 2.8]{dvr-puc} one approach to proving that for a topological group $\bG$ every (classical) $\bG$-space admits a universal $\bG$-equivariant Hausdorff compactification is to proceed categorically (via adjoint functor theorems in the style of \cite[\S V.6 Theorem 2 or \S V.8 Theorem 2]{mcl}), for instance. The present section applies the same category-theoretic approach to the non-commutative setting that is our focus.


Apart from $\cC^*_1$ and its various other flavors discussed above, we will encounter a number of other categories. 

Recall \cite[Introduction]{dvr-puc} that a {\it compactification} of a space $X$ is a compact Hausdorff space $K$ together with a dense-image continuous map $X\to K$. Casting possibly-non-unital $C^*$-algebras as non-commutative locally compact spaces, we work with the following analogue.

\begin{definition}\label{def:cpct}
  Let $A$ be a (possibly non-unital) $C^*$-algebra. A {\it compactification} of (the non-commutative space underlying) $A$ is a triple $(B,\ell,r)$ fitting into a commutative diagram
  \begin{equation}\label{eq:trip}
    \begin{tikzpicture}[auto,baseline=(current  bounding  box.center)]
      \path[anchor=base] 
      (0,0) node (l) {$A$}
      +(2,.5) node (u) {$B$}
      +(4,0) node (r) {$M(A)$}
      ;
      \draw[->] (l) to[bend left=6] node[pos=.5,auto] {$\scriptstyle \ell$} (u);
      \draw[->] (u) to[bend left=6] node[pos=.5,auto] {$\scriptstyle r$} (r);
      \draw[->] (l) to[bend right=6] node[pos=.5,auto,swap] {$\scriptstyle $} (r);
    \end{tikzpicture}
  \end{equation}
  where
  \begin{itemize}
  \item $r$ is a morphism in the category $\cC^*_1$ of unital $C^*$-algebras;
  \item $\ell$ is a morphism in $\cC^*$; 
  \item and the bottom arrow is the canonical inclusion. 
  \end{itemize}
  Compactifications of $A$ from a category $\cat{cpct}(A)$ in the obvious fashion, with morphisms
  \begin{equation*}
    (B,\ell,r)\to (B',\ell',r')
  \end{equation*}
  consisting of those unital $C^*$-morphisms $B\to B'$ making the appropriate four-vertex diagram commute.
\end{definition}

There are also equivariant versions of the various categories. 

\begin{definition}\label{def:cats}
  Let $\bG$ be an LCQG.
  \begin{enumerate}[(a)]
  \item\label{item:4} $\cC^{*\bG}_1$ is the category of unital $C^*$-algebras equipped with an action (\Cref{def:act}) by a fixed LCQG $\bG$, with $\bG$-equivariant unital morphisms in the sense of \Cref{def:eqvr}.

  \item\label{item:5} For an action $\rho:A\circlearrowleft \bG$ on a (possibly non-unital) $C^*$-algebra $A$ the category $\cat{cpct}(\rho)$ of {\it equivariant compactifications of $A$ (or $\rho$)} consists of quadruples $(B,\ell,r,\rho_B)$ where
    \begin{itemize}
    \item the triple $(B,\ell,r)$ is an object of $\cat{cpct}(A)$, as in \Cref{def:cpct};
    \item the unital $C^*$-algebra $B$ additionally carries a $\bG$ action $\rho_B:B\circlearrowleft \bG$;
    \item and all maps in \Cref{eq:trip} are $\bG$-equivariant in the sense of \Cref{def:eqvr}.
    \end{itemize}
    The morphisms are as in $\cat{cpct}(A)$, but also $\bG$-equivariant.
  \end{enumerate}  
\end{definition}

There are also the {\it comma categories} $A\downarrow \cC^*_1$ of \cite[\S II.6]{mcl}: the objects are $C^*$-morphisms $A\to B$ for unital $B$, with unital morphisms that make the relevant triangles commute.

The statement of \Cref{th:ccmpl} refers to the following notion, dual to \cite[\S V.1, Definition]{mcl} (slightly adapted for our use here).

\begin{definition}\label{def:crt}
  A functor $F:\cC\to \cC'$ {\it creates colimits} if for every functor $S:\cD\to \cC$ from a small diagram, a colimiting cocone for $F\circ S:\cD\to \cC'$ is the image of a unique (up to isomorphism) cocone on $S$, and that cocone is a colimit for $S$.
\end{definition}

\begin{theorem}\label{th:ccmpl}
  Let $\rho:A\circlearrowleft \bG$ be an LCQG action on a $C^*$-algebra. The three categories in the diagram
  \begin{equation}\label{eq:allforget}
    \begin{tikzpicture}[auto,baseline=(current  bounding  box.center)]
      \path[anchor=base] 
      (0,0) node (lu) {$\cat{cpct}(\rho)$}
      +(4,0) node (mu) {$\cat{cpct}(A)$}
      +(8,0) node (ru) {$A\downarrow \cC^*_1$}
      +(0,-1) node (ld) {$(B,\ell,r,\rho_B)$}
      +(4,-1) node (md) {$(B,\ell,r)$}
      +(8,-1) node (rd) {$A\stackrel{\ell}{\to} B$}
      ;
      \draw[->] (lu) to[bend left=6] node[pos=.5,auto] {$\scriptstyle \text{forget}$} (mu);
      \draw[->] (mu) to[bend left=6] node[pos=.5,auto] {$\scriptstyle \text{forget}$} (ru);
      \draw[|->] (ld) to[bend left=-6] node[pos=.5,auto] {$\scriptstyle $} (md);
      \draw[|->] (md) to[bend left=-6] node[pos=.5,auto] {$\scriptstyle $} (rd);
    \end{tikzpicture}
  \end{equation}
  are all cocomplete, and the forgetful functors depicted therein all create colimits.
\end{theorem}
\begin{proof}
  $\cC^*_1$ has coequalizers and coproducts \cite[\S 2.1]{ped-pp} (which discusses co{\it kernels}, but coequalizers are described similarly), so it is indeed cocomplete by, say, \cite[\S V.2, Theorem 1]{mcl} or \cite[Theorem 12.3]{ahs}. That its comma category
  \begin{equation*}
    A\downarrow \cC^*_1\simeq A^+\downarrow \cC^*_1
  \end{equation*}
  is also cocomplete is then a simple exercise: the colimit of a functor $S:\cD\to A^+\downarrow \cC^*_1$ is identifiable with the colimit of
  \begin{equation*}
    S':\cD'\to \cC^*_1,
  \end{equation*}
  where
  \begin{itemize}
  \item $\cD'$ is $\cD$ with an additional initial object $*$;
  \item and $S'$ restricts to $S$ on $\cD\subset \cD'$, sends $*\mapsto A^+$ and the unique morphisms $*\to d$ to the morphisms constituting the objects $S(d)\in A^+\downarrow \cC^*_1$.
  \end{itemize}
  
  We prove the colimit-creation claim for the overall forgetful functor
  \begin{equation*}
    U: \cat{cpct}(\rho)\to \cC^*_1,
  \end{equation*}
  i.e. the composition of all of the functors in the diagram \Cref{eq:allforget}. No additional subtleties emerge in working with the individual functors, so the relevant ingredients will all feature in this single proof.

  Denote $\cC:=\cat{cpct}(\rho)$ for brevity and consider a functor $S:\cD\to \cC$ from a small category $\cD$. The cocompleteness of $\cC^*_1$ ensures the existence of a colimit
  \begin{equation}\label{eq:bcolim}
    B:=\varinjlim US\in \cC^*_1,
  \end{equation}
  i.e. a unital $C^*$-algebra equipped with a colimiting $US$-cocone in $\cC^*_1$ consisting of unital $C^*$ morphisms
  \begin{equation*}
    \iota_d:B_d\to B,\quad\text{where}\ B_d:=S(d) \text{ for } d\in \cD. 
  \end{equation*}
  We will equip $B$ with all of the desired structure (an action by $\bG$, etc.), in the only way possible, given the various compatibility constraints.

  {\bf (1): a $\bG$-action $\rho:B\circlearrowleft \bG$.} Denoting by $\rho_d:B_d\circlearrowleft\bG$ the individual actions on the $C^*$-algebras $B_d$, we want an action $\rho:B\circlearrowleft \bG$ that fits into commutative diagrams
  \begin{equation}\label{eq:phideqvr}
    \begin{tikzpicture}[auto,baseline=(current  bounding  box.center)]
      \path[anchor=base] 
      (0,0) node (l) {$B_d$}
      +(4,.5) node (u) {$M(C_0(\bG)\otimes B_d)$}
      +(4,-.5) node (d) {$B$}
      +(8,0) node (r) {$M(C_0(\bG)\otimes B)$}
      ;
      \draw[->] (l) to[bend left=6] node[pos=.5,auto] {$\scriptstyle \rho_d$} (u);
      \draw[->] (u) to[bend left=6] node[pos=.5,auto] {$\scriptstyle \id\otimes\iota_d$} (r);
      \draw[->] (l) to[bend right=6] node[pos=.5,auto,swap] {$\scriptstyle \iota_d$} (d);
      \draw[->] (d) to[bend right=6] node[pos=.5,auto,swap] {$\scriptstyle \rho$} (r);
    \end{tikzpicture}
  \end{equation}
  for $d\in \cD$ and is further compatible in the obvious fashion with the morphisms
  \begin{equation*}
    S(f):B_d\to B_{d'}\text{ for }d\stackrel{f}{\to}d'\in D.
  \end{equation*}
  That such a $\varphi$ exists follows immediately from the universality property of the colimit \Cref{eq:bcolim}. It is of course unital and hence non-degenerate, and its coassociativity
  \begin{equation}\label{eq:rhoassoc}
    (\id\otimes\rho)\circ\rho = (\Delta\otimes\id)\circ\rho
  \end{equation}
  follows again from the universality property of the colimit \Cref{eq:bcolim} and the diagrams
    \begin{equation*}
    \begin{tikzpicture}[auto,baseline=(current  bounding  box.center)]
      \path[anchor=base] 
      (0,0) node (l) {$B_d$}
      +(4,1) node (u) {$M(C_0(\bG)\otimes C_0(\bG)\otimes B_d)$}
      +(4,-.5) node (d) {$B$}
      +(8,0) node (r) {$M(C_0(\bG)\otimes C_0(\bG)\otimes B)$,}
      ;
      \draw[->] (l) to[bend left=6] node[pos=.5,auto] {$\scriptstyle (\id\otimes\rho_d)\rho_d=(\Delta\otimes\id)\rho_d$} (u);
      \draw[->] (u) to[bend left=6] node[pos=.5,auto] {$\scriptstyle \id\otimes\id\otimes\iota_d$} (r);
      \draw[->] (l) to[bend right=6] node[pos=.5,auto,swap] {$\scriptstyle \iota_d$} (d);
      \draw[->] (d) to[bend right=6] node[pos=.5,auto,swap] {$\scriptstyle \bullet$} (r);
    \end{tikzpicture}
  \end{equation*}
  where either morphism in \Cref{eq:rhoassoc} will do to fill in the arrow marked `$\bullet$'.
  
  We also need to verify the {\it continuity} of the action $\rho:B\circlearrowleft \bG$, i.e. condition \Cref{item:8} of \Cref{def:act} with $B$ in place of $A$. This follows from the commutativity of \Cref{eq:phideqvr} and the analogous conditions
  \begin{equation*}
    \overline{\rho_d(B_d)(C_0(\bG)\otimes \bC)}^{\|\cdot \|} = C_0(\bG)\otimes B_d.
  \end{equation*}
  for the individual $\rho_d$, and the fact that
  \begin{equation*}
    \iota_d(B_d)\subseteq B,\ d\in D
  \end{equation*}
  generate $B$ as a $C^*$-algebra.

  {\bf (2): an equivariant morphism $\ell:A\to B$.} We already have equivariant morphisms $\ell_d:A\to B_d$, so simply setting
  \begin{equation*}
    \ell:=A\stackrel{\ell_d}{\longrightarrow} B_d\stackrel{\iota_d}{\longrightarrow} B
  \end{equation*}
  will do: the commutativity of the diagrams
  \begin{equation*}
    \begin{tikzpicture}[auto,baseline=(current  bounding  box.center)]
      \path[anchor=base] 
      (0,0) node (l) {$A$}
      +(2,.5) node (u) {$B_d$}
      +(4,0) node (r) {$B_{d'}$}
      ;
      \draw[->] (l) to[bend left=6] node[pos=.5,auto] {$\scriptstyle \ell_d$} (u);
      \draw[->] (u) to[bend left=6] node[pos=.5,auto] {$\scriptstyle S(f)$} (r);
      \draw[->] (l) to[bend right=6] node[pos=.5,auto,swap] {$\scriptstyle \ell_{d'}$} (r);
    \end{tikzpicture}
  \end{equation*}
  for $f:d\in d'$ ensures that this definition does not depend on $d$, and the equivariance follows from that of $\ell_d$ and $\iota_d$.

  {\bf (3): an equivariant morphism $r:B\to M(A)$.} We have equivariant morphisms $r_d$ fitting into commutative diagrams 
  \begin{equation*}
    \begin{tikzpicture}[auto,baseline=(current  bounding  box.center)]
      \path[anchor=base] 
      (0,0) node (l) {$B_d$}
      +(2,.5) node (u) {$B_{d'}$}
      +(4,0) node (r) {$M(A)$}
      ;
      \draw[->] (l) to[bend left=6] node[pos=.5,auto] {$\scriptstyle S(f)$} (u);
      \draw[->] (u) to[bend left=6] node[pos=.5,auto] {$\scriptstyle r_{d'}$} (r);
      \draw[->] (l) to[bend right=6] node[pos=.5,auto,swap] {$\scriptstyle r_d$} (r);
    \end{tikzpicture}
  \end{equation*}
  so the universality property of the colimit \Cref{eq:bcolim} provides $r:B\to M(A)$, along with its equivariance.

  {\bf (4): colimit creation.} That what we have just constructed is indeed a colimit in
  \begin{equation*}
    \cC=\cat{cpct}(\rho)
  \end{equation*}
  is obvious from the construction itself, which at every stage leverages a universality property. On the other hand, the {\it uniqueness} (up to isomorphism) of the resulting $S$-cocone required by \Cref{def:crt} is also a consequence of the selfsame universality of \Cref{def:crt}:
  \begin{itemize}
  \item the obligate commutativity of \Cref{eq:phideqvr} (for all $d\in D$) defines $\rho$ uniquely;
  \item similarly, the requisite commutativity of
    \begin{equation*}
      \begin{tikzpicture}[auto,baseline=(current  bounding  box.center)]
        \path[anchor=base] 
        (0,0) node (l) {$B_d$}
        +(2,.5) node (u) {$B$}
        +(4,0) node (r) {$M(A)$}
        ;
        \draw[->] (l) to[bend left=6] node[pos=.5,auto] {$\scriptstyle \iota_d$} (u);
        \draw[->] (u) to[bend left=6] node[pos=.5,auto] {$\scriptstyle r$} (r);
        \draw[->] (l) to[bend right=6] node[pos=.5,auto,swap] {$\scriptstyle r_d$} (r);
      \end{tikzpicture}
    \end{equation*}
    determines $r$;
  \item and once more,
    \begin{equation*}
      \begin{tikzpicture}[auto,baseline=(current  bounding  box.center)]
        \path[anchor=base] 
        (0,0) node (l) {$A$}
        +(2,.5) node (u) {$B_d$}
        +(4,0) node (r) {$B$}
        ;
        \draw[->] (l) to[bend left=6] node[pos=.5,auto] {$\scriptstyle \ell_d$} (u);
        \draw[->] (u) to[bend left=6] node[pos=.5,auto] {$\scriptstyle \iota_d$} (r);
        \draw[->] (l) to[bend right=6] node[pos=.5,auto,swap] {$\scriptstyle \ell$} (r);
      \end{tikzpicture}
    \end{equation*}
    returns a unique $\ell$. 
  \end{itemize}
  This finishes the proof. 
\end{proof}

It is perhaps worth spelling out an ``absolute'' version of \Cref{th:ccmpl}, involving $\bG$ alone (and no action $\rho$).

Note, incidentally, that the same arguments also give colimits and colimit creation absent $A$; we only record the statement, as the proof essentially recapitulates (in part) that of \Cref{th:ccmpl}.

\begin{corollary}\label{cor:ccmpl-noa}
  For any LCQG $\bG$ the category $\cC^{*\bG}_1$ is cocomplete and the forgetful functor
  \begin{equation*}
    \cC^{*\bG}_1\to \cC^*_1
  \end{equation*}
  creates colimits.
\end{corollary}
\begin{proof}
  Simply set $A=\{0\}$ in \Cref{th:ccmpl}. This is a unital $C^*$-algebra (in fact the {\it terminal object} of $\cC^*_1$ \cite[Definition 7.4]{ahs}, admitting exactly one morphism from every other object), and the right-hand forgetful functor in \Cref{eq:allforget} becomes an equivalence, while the left-hand functor specializes to $\cC^{*\bG}_1\to \cC^*_1$.
\end{proof}

\begin{remark}\label{re:clsdvr}
  \Cref{cor:ccmpl-noa} is very much in the spirit of \cite[Theorem 3.2.4]{dvr-bk}, to the effect that the forgetful functor from topological $\bG$-spaces to $\bG$-spaces creates {\it limits}: the categories of $C^*$-algebras under consideration here are intuitively dual to those of ``spaces''.
\end{remark}

\Cref{th:ccmpl} will help prove that among all compactifications of an action $\rho:A\circlearrowleft\bG$ there is a ``largest'' one, analogous to its classical counterpart discussed in \cite[\S 2.8]{dvr-puc}.

\begin{corollary}\label{cor:term}
  For an LCQG action $\rho$ the category $\cat{cpct}(\rho)$ has a terminal object.
\end{corollary}
\begin{proof}
  Given the cocompleteness provided by \Cref{th:ccmpl}, (the dual to) \cite[\S V.6, Theorem 1]{mcl} reduces the claim to the {\it solution-set condition}: there is a {\it set} (as opposed to a proper class) of objects
  \begin{equation*}
    x_i:=(B_i,\ell_i,r_i,\rho_{i})\in \cat{cpct}(\rho)
  \end{equation*}
  so that every object of $\cat{cpct}(\rho)$ admits a morphism to some $x_i$. We will argue that the $x_i$ can be chosen so that the right-hand morphisms $r_i:B_i\to M(A)$ are embeddings; clearly, then, they constitute a set, since there is an upper bound on the cardinalities of all of the structures involved.

  Consider an arbitrary object $(B,\ell,r,\rho_B)$ in $\cat{cpct}(\rho)$. The equivariance (\Cref{def:eqvr}) of $r:B\to M(A)$ provides the commutative diagram
  \begin{equation*}
    \begin{tikzpicture}[auto,baseline=(current  bounding  box.center)]
      \path[anchor=base] 
      (0,0) node (l) {$B$}
      +(4,.5) node (u) {$M(A)$}
      +(2,-1.5) node (dl) {$M(C_0(\bG)\otimes B)$}
      +(7,-1.5) node (dr) {$M(C_0(\bG)\otimes M(A))$}
      +(9,0) node (r) {$M(C_0(\bG)\otimes A)$}
      ;
      \draw[->] (l) to[bend left=6] node[pos=.5,auto] {$\scriptstyle r$} (u);
      \draw[->] (u) to[bend left=6] node[pos=.5,auto] {$\scriptstyle \rho_A$} (r);
      \draw[->] (l) to[bend right=6] node[pos=.5,auto,swap] {$\scriptstyle \rho_B$} (dl);
      \draw[->] (dl) to[bend right=6] node[pos=.5,auto,swap] {$\scriptstyle \id\otimes r$} (dr);
      \draw[->] (dr) to[bend right=6] node[pos=.5,auto,swap] {$\scriptstyle \subseteq $} (r);
    \end{tikzpicture}
  \end{equation*}
  whence it follows that the restriction of
  \begin{equation*}
    \rho_A:M(A)\to M(C_0(\bG)\otimes A)
  \end{equation*}
  to the unital $C^*$-subalgebra $r(B)\subseteq M(A)$ takes values in $M(C_0(\bG)\otimes r(B))$. That restriction, henceforth denoted by $\rho_{r(B)}$, then fits into a commutative diagram
  \begin{equation*}
    \begin{tikzpicture}[auto,baseline=(current  bounding  box.center)]
      \path[anchor=base] 
      (0,0) node (l) {$B$}
      +(3,.5) node (u) {$r(B)$}
      +(3,-.5) node (d) {$M(C_0(\bG)\otimes B)$}
      +(7,0) node (r) {$M(C_0(\bG)\otimes r(B))$}
      ;
      \draw[->] (l) to[bend left=6] node[pos=.5,auto] {$\scriptstyle r$} (u);
      \draw[->] (u) to[bend left=6] node[pos=.5,auto] {$\scriptstyle \rho_{r(B)}$} (r);
      \draw[->] (l) to[bend right=6] node[pos=.5,auto,swap] {$\scriptstyle \rho_B$} (d);
      \draw[->] (d) to[bend right=6] node[pos=.5,auto,swap] {$\scriptstyle \id\otimes r$} (r);
    \end{tikzpicture}
  \end{equation*}
  The requisite coassociativity of $\rho_{r(B)}$ follows immediately from that of $\rho_B$ and the surjectivity of $r:B\to r(B)$, as does the action-continuity condition
  \begin{equation*}
    \overline{\rho_{r(B)}(r(B))(C_0(\bG)\otimes \bC)} = C_0(\bG)\otimes r(B). 
  \end{equation*}
  We thus have a morphism in $\cat{cpct}(\rho)$ from the arbitrary object $(B,\ell,r,\rho_B)$ to the object
  \begin{equation*}
    (r(B),\ r\circ\ell,\ (r(B)\subseteq M(A)),\ \rho_{r(B)})\in \cat{cpct}(\rho). 
  \end{equation*}
  As explained above, the fact that the third component of this new object is an embedding into $M(A)$ completes the proof. 
\end{proof}

We can now define, in the present quantum context, versions of classical action-compactifications (see \cite[\S 2.8]{dvr-puc} for arbitrary actions and \cite[Introduction]{brk} for compactifications of regular self-actions, i.e. {\it ambits}).

\begin{definition}\label{def:univcpct}
  The {\it greatest (or maximal, or equivariant) compactification} of an LCQG action $\rho:A\circlearrowleft\bG$ is the terminal object of $\cat{cpct}(\rho)$ provided by \Cref{cor:term}.

  The {\it greatest ambit} of a locally compact quantum group $\bG$ is the maximal compactification of its regular action
  \begin{equation*}
    \Delta_{\bG}:C_0(\bG)\to M(C_0(\bG)\otimes C_0(\bG))
  \end{equation*}
  discussed in \Cref{ex:regred}. 
\end{definition}

\section{Equivariant epimorphisms and regular monomorphisms}\label{se:epimono}

Recall \cite[Definitions 7.32 and 7.39]{ahs} that a morphism $f:A\to B$ in a category is
\begin{itemize}
\item {\it monic} (or a {\it monomorphism}) if it is left-cancellable: for every parallel pair $f_i:A'\to A$, $i=1,2$ we have
  \begin{equation*}
    ff_1 = ff_2\Rightarrow f_1=f_2.
  \end{equation*}
\item dually, {\it epic} (or an {\it epimorphism}) if it is right-cancellable: for every parallel pair $f_i:B\to B'$, $i=1,2$ we have
  \begin{equation*}
    f_1 f = f_2 f\Rightarrow f_1=f_2. 
  \end{equation*}
\end{itemize}
The notions are meant to recast, in purely categorical terms, those of injection and surjection respectively.

In {\it concrete} categories \cite[Definition 5.1]{ahs} such as those of \Cref{th:ccmpl} (i.e. those consisting of sets with additional structure, where morphisms are structure-respecting functions) surjections are always epimorphisms, but the converse is occasionally interesting and non-trivial: \cite[Theorem 2]{reid} and \cite[Corollary 4]{hnb}, for instance, prove this converse (epimorphisms are surjective) for $\cC^*_1$. We extend the result to the categories studied in \Cref{th:ccmpl}, after some preparation.

\begin{theorem}\label{th:episurj}
  For any locally compact quantum group $\bG$ and action $\rho:A\circlearrowleft \bG$, the epimorphisms in any of the categories $\cat{cpct}(\rho)$, $\cat{cpct}(A)$ and $A\downarrow \cC^*_1$ are those whose underlying $C^*$-morphism is surjective.
\end{theorem}

Recall first the following well-known characterization of monomorphisms.

\begin{lemma}\label{le:monchar}
  Let $f:A\to B$ be a morphism in a category, and assume its self-pullback
  \begin{equation*}
    \begin{tikzpicture}[auto,baseline=(current  bounding  box.center)]
      \path[anchor=base] 
      (0,0) node (l) {$A\times_BA$}
      +(2,.5) node (u) {$A$}
      +(2,-.5) node (d) {$A$}
      +(4,0) node (r) {$B$}
      ;
      \draw[->] (l) to[bend left=6] node[pos=.5,auto] {$\scriptstyle \pi_1$} (u);
      \draw[->] (u) to[bend left=6] node[pos=.5,auto] {$\scriptstyle f$} (r);
      \draw[->] (l) to[bend right=6] node[pos=.5,auto,swap] {$\scriptstyle \pi_2$} (d);
      \draw[->] (d) to[bend right=6] node[pos=.5,auto,swap] {$\scriptstyle f$} (r);
    \end{tikzpicture}
  \end{equation*}
  exists. The following are equivalent:
  \begin{enumerate}[(a)]
  \item\label{item:17} $f$ is monic;
  \item\label{item:18} $\pi_1$ is monic;
  \item\label{item:20} both $\pi_i$ are monic;
  \item\label{item:19} $\pi_1$ is an isomorphism;  
  \item\label{item:21} both $\pi_i$ are isomorphisms.
  \end{enumerate}
\end{lemma}
\begin{proof}
  That \Cref{item:18} $\Leftrightarrow$ \Cref{item:20} and \Cref{item:19} $\Leftrightarrow$ \Cref{item:21} follows from the fact that $A\times_BA$ has an automorphism interchanging the factors and the $\pi_i$, the equivalence \Cref{item:17} $\Leftrightarrow$ \Cref{item:21} is noted in \cite[\S I.2, discussion on kernel pairs]{mm}, and given that $\pi_1$ is a {\it retraction} \cite[Definition 7.24]{ahs} (i.e. it has a right inverse, namely the map $A\to A\times_BA$ with identity components), it is an isomorphism precisely when it is monic \cite[Proposition 7.36]{ahs}.
\end{proof}

Naturally, there is a dual version of \Cref{le:monchar} relating epimorphisms and pushouts, etc. As a consequence, we have

\begin{lemma}\label{le:prespp}
  \begin{enumerate}[(a)]
  \item Pullback-preserving functors defined on categories with pullbacks also preserve monomorphisms.
  \item Dually, pushout-preserving functors defined on categories with pushouts preserve epimorphisms.
  \end{enumerate}
\end{lemma}
\begin{proof}
  Immediate from \Cref{le:monchar} and its dual version. 
\end{proof}

\pf{th:episurj}
\begin{th:episurj}
  \Cref{th:ccmpl} and \Cref{le:prespp} reduce the problem to the rightmost category in \Cref{eq:allforget}, i.e. the comma category $A\downarrow \cC^*_1$. Consider an epimorphism in that category, consisting of a morphism $f:B\to C$ in $\cC^*_1$ fitting into a commutative triangle
  \begin{equation*}
    \begin{tikzpicture}[auto,baseline=(current  bounding  box.center)]
      \path[anchor=base] 
      (0,0) node (l) {$A$}
      +(2,.5) node (u) {$B$}
      +(4,0) node (r) {$C$}
      ;
      \draw[->] (l) to[bend left=6] node[pos=.5,auto] {$\scriptstyle $} (u);
      \draw[->] (u) to[bend left=6] node[pos=.5,auto] {$\scriptstyle f$} (r);
      \draw[->] (l) to[bend right=6] node[pos=.5,auto,swap] {$\scriptstyle $} (r);
    \end{tikzpicture}
  \end{equation*}
  Factoring $f=f'f''$ through its image, $f'$ is also epic \cite[Proposition 7.41]{ahs}; there is no loss of generality, then, in assuming that $f$ itself is an embedding. As such, it must be an equalizer in $\cC^*_1$ by \cite[Theorem 6]{hnb}. Because for any object $c\in \cC$ in complete category the forgetful functor
  \begin{equation*}
    c\downarrow \cC\to \cC
  \end{equation*}
  is easily seen to create limits, $f$ is also an equalizer in
  \begin{equation*}
    A\downarrow \cC^*_1\simeq A^+\downarrow \cC^*_1.
  \end{equation*}
  An epic equalizer is an isomorphism \cite[Proposition 7.54]{ahs}, so we are done.
\end{th:episurj}

As noted in the proof of \Cref{th:episurj}, \cite[Theorem 6]{hnb} shows that embeddings of $C^*$-algebras are equalizers. It is not difficult to show that, dually to \cite[Theorem 2]{reid} or \cite[Corollary 4]{hnb}, in $\cC^*_1$ the {\it monomorphisms} (or {\it monic} morphisms \cite[Definition 7.32]{ahs}; the notion dual to that of `epimorphism') are exactly the embeddings.

Going back to monomorphisms in $\cC^*_1$,
\begin{itemize}
\item observe that the forgetful functor $\cC^*_1\to\cat{Set}$ preserves finite limits;
\item so the pullback of a $\cC^*_1$-morphism $f:A\to B$ along itself is simply the set-theoretic pullback $A\times_BA$;
\item whence the two projections $A\times_BA\to A$ can be isomorphisms (i.e. $f$ is monic, per \Cref{le:monchar}) only when $A\to B$ is one-to-one.
\end{itemize}
In view of this, \cite[Theorem 6]{hnb} says that in $\cC^*_1$ monomorphisms are {\it regular} \cite[\S 0.5]{ar}.

I do not know whether monomorphisms in any of the categories from \Cref{th:ccmpl} are injective (more on this below), but \cite[Theorem 6]{hnb} does generalize: those that {\it do} happen to be injective are equalizers.

\begin{theorem}\label{th:injeq}
  Let $\bG$ be a locally compact group and $\rho:A\circlearrowleft \bG$ an action. In any of the categories $\cat{cpct}(\rho)$, $\cat{cpct}(A)$ and $A\downarrow \cC^*_1$, the morphisms whose underlying $C^*$-algebra maps are injective are equalizers.
\end{theorem}
\begin{proof}
  Consider a morphism $\iota:x\to y$ in any of the categories in question, mapped to an injection by the forgetful functor to $\cC^*_1$. We saw, in the course of the proof of \Cref{th:episurj}, that $\iota$ is an equalizer in $A\downarrow \cC^*_1$ (by forgetting to $\cC^*_1$, etc.). But then it is also the equalizer (in $A\downarrow \cC^*_1$ again) of its own self-pushout
  \begin{equation*}
    \begin{tikzpicture}[auto,baseline=(current  bounding  box.center)]
      \path[anchor=base] 
      (0,0) node (l) {$x$}
      +(2,.5) node (u) {$y$}
      +(2,-.5) node (d) {$y$}
      +(4,0) node (r) {$y\coprod_xy$:}
      ;
      \draw[->] (l) to[bend left=6] node[pos=.5,auto] {$\scriptstyle \iota$} (u);
      \draw[->] (u) to[bend left=6] node[pos=.5,auto] {$\scriptstyle \iota_1$} (r);
      \draw[->] (l) to[bend right=6] node[pos=.5,auto,swap] {$\scriptstyle \iota$} (d);
      \draw[->] (d) to[bend right=6] node[pos=.5,auto,swap] {$\scriptstyle \iota_2$} (r);
    \end{tikzpicture}
  \end{equation*}
  this is the observation dual to \cite[Proposition 11.33]{ahs}. Now, since by \Cref{th:ccmpl} the forgetful functors from any of our categories down to $A\downarrow \cC^*_1$ all create colimits, the equalizer 
  \begin{equation*}
    \begin{tikzpicture}[auto,baseline=(current  bounding  box.center)]
      \path[anchor=base] 
      (0,0) node (l) {$x$}
      +(2,0) node (u) {$y$}
      +(4,0) node (r) {$y\coprod_xy$}
      ;
      \draw[->] (l) to[bend left=0] node[pos=.5,auto] {$\scriptstyle \iota$} (u);
      \draw[->] (u) to[bend left=6] node[pos=.5,auto] {$\scriptstyle \iota_1$} (r);      
      \draw[->] (u) to[bend right=6] node[pos=.5,auto,swap] {$\scriptstyle \iota_2$} (r);
    \end{tikzpicture}
  \end{equation*}
  can be reinterpreted as one in $\cat{cpct}(\rho)$ or $\cat{cpct}(A)$, as appropriate. 
\end{proof}

\section{Local presentability}\label{se:lpres}

We recalled in the course of the proof of \Cref{th:ccmpl} that the category $\cC^*_1$ of unital $C^*$-algebras is cocomplete; it is in fact a great deal more than that: as noted in \cite[Remark 6.10]{dlros}, $\cC^*_1$ is {\it locally $\aleph_1$-presentable} in the sense of \cite[Definition 1.17]{ar}. This follows, for instance, from \cite[Theorem 3.28]{ar} and the fact that $\cC^*_1$ is a {\it variety of algebras} equipped with $\aleph_0$-ary operations \cite[Theorem 2.4]{pr1}.

We recall some category-theoretic background from \cite{ar}, assuming some of the more basic material (e.g. (co)limits) covered, say, in \cite{mcl}. Recall \cite[paragraph preceding Definition 1.13]{ar} that a {\it regular} cardinal $\kappa$ is one which cannot be written as a union of fewer than $\kappa$ sets, each of cardinality $<\kappa$.

We will often think of posets $(I,\le)$ as categories, with one arrow $i\to j$ whenever $i\le j$ (as in \cite[\S I.2]{mcl}). An aggregate of \cite[Definitions 1.13 and 1.17]{ar} now reads

\begin{definition}\label{def:lpres}
  Let $\kappa$ be a regular cardinal. A poset is {\it $\kappa$-directed} if every set of $<\kappa$ elements has an upper bound.
  
  Let $\cC$ be a category.
  \begin{enumerate}[(a)]
  \item A {\it $\kappa$-directed diagram} in $\cC$ is a functor $(I,\le)\to \cC$ for a $\kappa$-directed poset $(I,\le)$.
  \item A {\it $\kappa$-directed colimit} in $\cC$ is the colimit of a $\kappa$-directed diagram.
  \item An object $c\in \cC$ is {\it $\kappa$-presentable} if $\mathrm{hom}(c,-)$ preserves $\kappa$-directed colimits, and it is {\it presentable} if it is $\kappa$-presentable for some regular cardinal $\kappa$.
  \item $\cC$ is {\it $\kappa$-presentable} if
    \begin{itemize}
    \item it is cocomplete (i.e. has arbitrary small colimits; dual to \cite[\S V.1]{mcl});
    \item and has a set $S$ of $\kappa$-presentable objects such that every object is a $\kappa$-directed colimit of objects in $S$.
    \end{itemize}
  \item $\cC$ is {\it locally presentable} if it is locally $\kappa$-presentable for some regular cardinal $\kappa$.
  \end{enumerate}
  We might, on occasion, drop the modifier `locally' and speak instead of `presentable categories', etc.
\end{definition}

\begin{remark}
  As observed, $\cC^*_1$ is locally $\aleph_1$-presentable, hence its comma category
  \begin{equation*}
    A\downarrow \cC^*_1\simeq A^+\downarrow \cC^*_1
  \end{equation*}
  must also be $\aleph_1$-presentable \cite[Proposition 1.57]{ar}.
\end{remark}

Presentability is of interest here in part because it facilitates existence proofs for various universal structures: limits, adjoint functors, etc. Per \cite[Remark 1.56]{ar}, a presentable category is
\begin{itemize}
\item complete (in addition to being, by definition, {\it co}complete);
\item {\it well-powered} (i.e. every object has only a set of subobjects \cite[\S V.8]{mcl}) and dually, {\it co-well-powered}. 
\end{itemize}

All of this renders presentable categories well-behaved enough to automatically satisfy the requirements of {\it Freyd's Special Adjoint Functor Theorem (SAFT)} \cite[\S V.8, Theorem 2]{mcl}: in the present context, the dual of that result reads

\begin{proposition}\label{pr:saft}
  A functor defined on a locally presentable category is a left adjoint if and only if it is cocontinuous.  \qedhere
\end{proposition}

\begin{theorem}\label{th:allpres}
  For an LCQG action $\rho:A\circlearrowleft \bG$ the categories $\cat{cpct}(\rho)$, $\cat{cpct}(A)$ and $A\downarrow\cC^{*\bG}_1$ are all locally presentable.
\end{theorem}

First, the consequences noted above:

\begin{corollary}\label{cor:allgood}
  For an LCQG action $\rho:A\circlearrowleft \bG$ the categories $\cat{cpct}(\rho)$, $\cat{cpct}(A)$ and $A\downarrow\cC^{*\bG}_1$ are complete, well-powered and co-well-powered.
\end{corollary}
\begin{proof}
  Immediate from \Cref{th:allpres} and \cite[Remark 1.56]{ar}.
\end{proof}

And again:

\begin{corollary}\label{cor:ladj}
  The forgetful functors in \Cref{eq:allforget} are left adjoints.
\end{corollary}
\begin{proof}
  This follows from \Cref{th:ccmpl,th:allpres} and \Cref{pr:saft}.
\end{proof}

Rather than attack \Cref{th:allpres} as-is, we make some preparations. These are mostly intended to avoid extensive work with directed diagrams of non-injective $C^*$ morphisms, which do not play well with minimal $C^*$-algebra tensor products: \cite[\S II.9.6.5]{blk}, for instance, notes that the endofunctor $C\otimes-$ of $\cC^*$ preserves directed colimits of injections but not arbitrary directed colimits.

To mitigate the problem we rely on the theory of {\it locally generated} categories, as covered in \cite[\S 1.E]{ar} and extended in \cite[Definition]{ar-lg}. The latter source generalizes the narrower concept of the former in the context of a {\it factorization system}, which notion we recollect.

\begin{definition}\label{def:fact}
  A {\it factorization system} on a category $\cC$ is a pair $(\cE,\cM)$ of classes of morphisms in $\cC$ such that
  \begin{enumerate}[(1)]
  \item $\cE$ and $\cM$ are both closed under composition;
  \item the isomorphisms of $\cC$ are contained in both $\cE$ and $\cM$;
  \item and $\cC$ has {\it (essentially) unique $(\cE,\cM)$-factorizations} in the sense that every morphism $f$ factors as
    \begin{equation*}
      f = m\circ e,\ m\in \cM,\ e\in \cE,
    \end{equation*}
    and this factorization is unique up to unique isomorphism: two such factorizations on the outside of the diagram 
    \begin{equation*}
      \begin{tikzpicture}[auto,baseline=(current  bounding  box.center)]
        \path[anchor=base] 
        (0,0) node (l) {$\bullet$}
        +(2,.5) node (u) {$\bullet$}
        +(2,-.5) node (d) {$\bullet$}
        +(4,0) node (r) {$\bullet$}
        ;
        \draw[->] (l) to[bend left=6] node[pos=.5,auto] {$\scriptstyle e_1$} (u);
        \draw[->] (u) to[bend left=6] node[pos=.5,auto] {$\scriptstyle m_1$} (r);
        \draw[->] (l) to[bend right=6] node[pos=.5,auto,swap] {$\scriptstyle e_2$} (d);
        \draw[->] (d) to[bend right=6] node[pos=.5,auto,swap] {$\scriptstyle m_2$} (r);
        \draw[->] (u) to[bend right=0] node[pos=.5,auto,swap] {$\scriptstyle h$} node[pos=.5,auto] {$\scriptstyle \cong$} (d);
      \end{tikzpicture}
    \end{equation*}
    admit a unique vertical isomorphism making the diagram commute.
  \end{enumerate}

  Given morphism properties $\cP$ and $\cQ$, a factorization system $(\cE,\cM)$ is {\it $(\cP,\cQ)$} if $\cE$ consists of morphisms having property $\cP$ and $\cM$ consists of morphisms having property $\cQ$. Example: {\it (epi, mono)} factorization systems are those in which $\cE$ consists of epimorphisms and $\cM$ of monomorphisms.
\end{definition}

\begin{remark}
  The apparently-stronger notion of factorization system of \cite[Definition 14.1]{ahs} is equivalent by \cite[Proposition 14.7]{ahs}. On a related note, \cite[Proposition 14.6]{ahs} shows that in fact the class of isomorphisms precisely {\it coincides} with $\cE\cap \cM$ (rather than only being contained therein, as initially assumed).
\end{remark}

(Epi, mono) factorization systems in the sense of \Cref{def:fact} are essentially the gadgets introduced in \cite[\S III]{isbell-fact} modulo different language and coincide with the notion defined in passing in \cite[discussion preceding Definition]{ar-lg}. We will work mostly with (epi, mono) factorizations.

The alternative take on locally presentable categories, given a factorization system $(\cE,\cM)$, is as follows (\cite[Definition]{ar-lg} and also \cite[Definition 2.5]{dlros}):

\begin{definition}\label{def:lg}
  Let $(\cE,\cM)$ be a factorization system on a category $\cC$ and $\kappa$ a regular cardinal.
  \begin{enumerate}[(a)]
  \item\label{item:11} An object $c\in \cC$ is {\it $\cM$-$\kappa$-generated} or {\it $\kappa$-generated with respect to (wrt) $\cM$} if $\mathrm{hom}(c,-)$ preserves $\kappa$-directed colimits with connecting morphisms in $\cM$.
  \item\label{item:12} $c$ is {\it $\cM$-generated} or {\it generated with respect to $\cM$} if it is $\cM$-$\kappa$-generated for some regular cardinal $\kappa$.
  \item\label{item:13} $\cC$ is {\it $\cM$-locally $\kappa$-generated} if
    \begin{itemize}
    \item it is cocomplete;
    \item and has a set $S$ of $\cM$-$\kappa$-generated objects such that every object is a $\kappa$-directed colimit of objects in $S$ with connecting morphisms in $\cM$.
    \end{itemize}
  \item\label{item:14} $\cC$ is {\it $\cM$-locally generated} if it is $\cM$-locally $\kappa$-generated for some regular cardinal $\kappa$.
  \end{enumerate}
\end{definition}

\cite[Theorem 1]{ar-lg} says that local presentability is equivalent to $\cM$-local generation for {\it any} factorization system $(\cE,\cM)$. To bring that result in scope, then, we need such factorization systems.

\begin{proposition}\label{pr:surjinj}
  For any locally compact quantum group $\bG$, the category $\cC^{*\bG}_1$ has a factorization system $(\cE,\cM)$ where
  \begin{itemize}
  \item $\cE$ is the class of surjective morphisms in $\cC^{*\bG}_1$;
  \item and $\cM$ is the class of injective morphisms.
  \end{itemize}
\end{proposition}
\begin{proof}
  The first two conditions in \Cref{def:fact} are self-evident: composition preserves both surjectivity and injectivity, and isomorphisms enjoy both properties; it thus remains to discuss (unique) factorization.

  Let $\rho_A:A\circlearrowleft \bG$ and $\rho_B:B\circlearrowleft\bG$ be two objects in $\cC^{*\bG}_1$. A $\bG$-equivariant unital-$C^*$-algebra morphism $f:A\to B$ factors as
  \begin{equation}\label{eq:ffact}
    \begin{tikzpicture}[auto,baseline=(current  bounding  box.center)]
      \path[anchor=base] 
      (0,0) node (l) {$A$}
      +(2,.5) node (u) {$C$}
      +(4,0) node (r) {$B$}
      ;
      \draw[->>] (l) to[bend left=6] node[pos=.5,auto] {$\scriptstyle \pi$} (u);
      \draw[right hook->] (u) to[bend left=6] node[pos=.5,auto] {$\scriptstyle \iota$} (r);
      \draw[->] (l) to[bend right=6] node[pos=.5,auto,swap] {$\scriptstyle f$} (r);
    \end{tikzpicture}
  \end{equation}
  through its image as a surjection followed by an embedding, so we will have factorization as soon as we equip $C$ with a $\bG$-action making both $\pi$ and $\iota$ equivariant.

  By \Cref{le:isfnc} we have an embedding
  \begin{equation*}
    \id\otimes\iota:M(C_0(\bG)\otimes C)\subseteq M(C_0(\bG)\otimes B)
  \end{equation*}
  (we can drop the subscripts on the multiplier-algebra symbols by \Cref{re:when1} because $C$, $B$, etc. are unital). The commutativity of the diagram
  \begin{equation*}
    \begin{tikzpicture}[auto,baseline=(current  bounding  box.center)]
      \path[anchor=base] 
      (0,0) node (l) {$A$}
      +(2,.5) node (u) {$M(C_0(\bG)\otimes A)$}
      +(3,-.5) node (dl) {$C$}
      +(6,-.5) node (d) {$B$}
      +(6,.5) node (ur) {$M(C_0(\bG)\otimes C)$}
      +(10,0) node (r) {$M(C_0(\bG)\otimes B)$,}
      ;
      \draw[->] (l) to[bend left=6] node[pos=.5,auto] {$\scriptstyle \rho_A$} (u);
      \draw[->] (u) to[bend left=6] node[pos=.5,auto] {$\scriptstyle \id\otimes\pi$} (ur);
      \draw[right hook->] (ur) to[bend left=6] node[pos=.5,auto] {$\scriptstyle \id\otimes\iota$} (r);
      \draw[->>] (l) to[bend right=6] node[pos=.5,auto,swap] {$\scriptstyle \pi$} (dl);
      \draw[right hook->] (dl) to[bend right=6] node[pos=.5,auto,swap] {$\scriptstyle \iota$} (d);
      \draw[->] (d) to[bend right=6] node[pos=.5,auto,swap] {$\scriptstyle \rho_B$} (r);
    \end{tikzpicture}
  \end{equation*}
  expressing the equivariance of $f$, ensures that $\rho_B$ restricts to a (unital, hence non-degenerate) morphism
  \begin{equation*}
    \rho_C:C\to M(C_0(\bG)\otimes C)
  \end{equation*}
  so that furthermore $\pi:A\to C$ is $\rho_A$-$\rho_C$-equivariant:
  \begin{equation}\label{eq:piequiv}
    \begin{tikzpicture}[auto,baseline=(current  bounding  box.center)]
      \path[anchor=base] 
      (0,0) node (l) {$A$}
      +(2,.5) node (u) {$M(C_0(\bG)\otimes A)$}
      +(4,-.5) node (d) {$C$}
      +(6,0) node (r) {$M(C_0(\bG)\otimes C)$.}
      ;
      \draw[->] (l) to[bend left=6] node[pos=.5,auto] {$\scriptstyle \rho_A$} (u);
      \draw[->] (u) to[bend left=6] node[pos=.5,auto] {$\scriptstyle \id\otimes\pi$} (r);
      \draw[->] (l) to[bend right=6] node[pos=.5,auto,swap] {$\scriptstyle \pi$} (d);
      \draw[->] (d) to[bend right=6] node[pos=.5,auto,swap] {$\scriptstyle \rho_C$} (r);
    \end{tikzpicture}
  \end{equation}
  The coassociativity of $\rho_C$ follows from that of $\rho_B$ restricted to $C$, together with the fact that
  \begin{equation*}
    \id^{\otimes 2}\otimes \iota: M(C_0(\bG)^{\otimes 2}\otimes C)\to M(C_0(\bG)^{\otimes 2}\otimes B)
  \end{equation*}
  is an embedding (\Cref{le:isfnc} again). Finally, the continuity condition (\Cref{def:act} \Cref{item:8})
  \begin{equation*}
    \overline{\rho_C(C)(C_0(\bG)\otimes \bC)}^{\|\cdot \|} = C_0(\bG)\otimes C
  \end{equation*}
  follows from its analogue for $A$, the commutativity of \Cref{eq:piequiv}, and the surjectivity of $\pi$.

  Thus far we have a factorization
  \begin{equation*}
    f=m\circ e,\ m\in \cM,\ e\in \cE
  \end{equation*}
  for an arbitrary morphism \Cref{eq:ffact} in $\cC^{*\bG}_1$. Uniqueness is clear: for every such factorization the image of $e$ can be identified with the image $C\subseteq B$ of $f$ as above, and the $\bG$-action on that image is then uniquely defined by the requirement that
  \begin{equation*}
    \begin{tikzpicture}[auto,baseline=(current  bounding  box.center)]
      \path[anchor=base] 
      (0,0) node (l) {$C$}
      +(2,.5) node (u) {$M(C_0(\bG)\otimes C)$}
      +(4,-.5) node (d) {$B$}
      +(6,0) node (r) {$M(C_0(\bG)\otimes B)$.}
      ;
      \draw[->] (l) to[bend left=6] node[pos=.5,auto] {$\scriptstyle \rho_C$} (u);
      \draw[right hook->] (u) to[bend left=6] node[pos=.5,auto] {$\scriptstyle \id\otimes\iota$} (r);
      \draw[right hook->] (l) to[bend right=6] node[pos=.5,auto,swap] {$\scriptstyle \iota$} (d);
      \draw[->] (d) to[bend right=6] node[pos=.5,auto,swap] {$\scriptstyle \rho_B$} (r);
    \end{tikzpicture}
  \end{equation*}
  commute.
\end{proof}

\pf{th:allpres}
\begin{th:allpres}
  $\cat{cpct}(A)$ is a particular instance of $\cat{cpct}(\rho)$ (for trivial $\bG$). On the other hand, this latter category can be recovered from $\cC^{*\bG}_1$ through a double iteration of the comma-category construction:
  \begin{equation*}
    \cat{cpct}(\rho)\simeq \left(A^+\downarrow \cC^{*\bG}_1\right) \downarrow \left(A^+\to M(A)\right),
  \end{equation*}
  where $A^+$ is regarded as an object in $\cC^{*\bG}_1$ in the obvious fashion. Since local presentability transports over to comma categories \cite[Proposition 1.57]{ar}, it will suffice to focus on $\cC^{*\bG}_1$.

  We already know from \Cref{cor:ccmpl-noa} that $\cC^{*\bG}_{1}$ is cocomplete, so by \cite[Theorem 1]{ar-lg} it will be enough to argue that it is also $\cM$-locally generated for the (surjection,injection) factorization system $(\cE,\cM)$ of \Cref{pr:surjinj}. The claim is twofold.
  \begin{enumerate}[(1)]
  \item\label{item:9} For some regular cardinal $\kappa$ depending only on $\bG$, every object in $\cC^{*\bG}_1$ is the closed $\kappa$-directed union of its subobjects generated, as $C^*$-algebras, by fewer than $\kappa$ elements; we refer to these as {\it $\kappa$-capped} objects.
  \item\label{item:10} Every object of $\cC^{*\bG}_1$ is $\cM$-generated in the sense of \Cref{def:lg} \Cref{item:12}.
  \end{enumerate}
  Given \Cref{item:9} and \Cref{item:10}, the conclusion follows: for $\kappa$ as in \Cref{item:9}, there is, by \Cref{item:10}, a regular cardinal $\kappa'>\kappa$ such that all $\kappa$-capped objects are $\kappa'$-presentable (we can always raise the index of generation by \cite[Definition 1.13, last paragraph]{ar}). But then $\cC^{*\bG}_1$ is $\cM$-locally $\kappa'$-generated by definition, and we are done.

  It thus remains to prove \Cref{item:9} and \Cref{item:10}. The former is relegated to \Cref{le:univk}, so the rest of the proof is devoted to the latter: every $\cC^{*\bG}_1$-object is generated with respect to the class of injective morphisms.

  Let
  \begin{equation}\label{eq:atocolim}
    A\to B:=\varinjlim_i B
  \end{equation}
  be a morphism in $\cC^{*\bG}_1$, where the colimit (again in $\cC^{*\bG}_1$) on the right-hand side is $\kappa$-directed and its connecting morphisms $B_i\to B_j$, $i\le j$ are injective.

  We have already noted that $\cC^*_1$ is locally $\aleph_1$-presentable, so every object therein is presentable \cite[Remark 1.30 (1)]{ar}. It follows that for some $i$, the morphism \Cref{eq:atocolim} factors (uniquely, by the injectivity of $B_i\to B$) through $B_i$. We will be done once we show that the resulting ($C^*$, thus far) morphism $A\to B_i$ is $\bG$-equivariant. This, though, follows from the fact that \Cref{eq:atocolim} itself is equivariant by assumption, and hence
  \begin{equation*}
    \begin{tikzpicture}[auto,baseline=(current  bounding  box.center)]
      \path[anchor=base] 
      (0,0) node (l) {$A$}
      +(2,.5) node (u) {$M(C_0(\bG)\otimes A)$}
      +(4,-.5) node (d) {$B_i$}
      +(6,0) node (r) {$M(C_0(\bG)\otimes B_i)$}
      ;
      \draw[->] (l) to[bend left=6] node[pos=.5,auto] {$\scriptstyle $} (u);
      \draw[->] (u) to[bend left=6] node[pos=.5,auto] {$\scriptstyle $} (r);
      \draw[->] (l) to[bend right=6] node[pos=.5,auto,swap] {$\scriptstyle $} (d);
      \draw[->] (d) to[bend right=6] node[pos=.5,auto,swap] {$\scriptstyle $} (r);
    \end{tikzpicture}
  \end{equation*}
  commutes after further composition with the {\it injection} (\Cref{le:isfnc})
  \begin{equation*}
    M(C_0(\bG)\otimes B_i)\to M(C_0(\bG)\otimes B).
  \end{equation*}
  This completes the proof of the theorem, modulo \Cref{le:univk}. 
\end{th:allpres}

\begin{lemma}\label{le:univk}
  Let $\bG$ be a locally compact quantum group. There is a regular cardinal $\kappa=\kappa(\bG)$ such that every object
  \begin{equation*}
    \rho:A\circlearrowleft\bG\in \cC^{*\bG}_1
  \end{equation*}
  is the closed $\kappa$-directed union of its subobjects generated, as $C^*$-algebras, by fewer than $\kappa$ elements.
\end{lemma}
\begin{proof}
  We want to produce, for each $a\in A$, a unital $C^*$-subalgebra $a\in B\subseteq A$ generated by fewer than $\kappa$ elements such that 
  \begin{enumerate}[(a)]
  \item\label{item:15} the action structure map $\rho:A\to M_A(C_0(\bG)\otimes A)$ maps $B$ to
    \begin{equation*}
      M_B(C_0(\bG)\otimes B)\subseteq M_A(C_0(\bG)\otimes A)
    \end{equation*}
    (making implicit use of the embedding claim in \Cref{le:isfnc});
  \item\label{item:16} and furthermore, the action-continuity condition
    \begin{equation*}
      \overline{\rho(B)(C_0(\bG)\otimes \bC)}^{\|\cdot \|} = C_0(\bG)\otimes B
    \end{equation*}
    is met.
  \end{enumerate}
  We will start with the unital $C^*$-subalgebra of $A$ generated by $a$ and enlarge it recursively. As it will be apparent from the construction that the procedure produces $C^*$-algebras generated by $<\kappa$ elements for an appropriately large cardinal $\kappa$ depending only on $\bG$, we henceforth omit any mention of cardinals, referring only to ``small'' $C^*$-subalgebras; this means generated by $<\kappa$ elements for some large but fixed cardinal number depending on nothing but $\bG$.

  The above-mentioned enlargement procedure is recursive, taking turns in ensuring we have \Cref{item:15} and \Cref{item:16} and passing to a colimit. The two conditions are each provided by a separate recursive construction; we describe these in turn. 

  {\bf Step (1): enlarging a small $B'\subseteq A$ to a small $B\subseteq A$ satisfying \Cref{item:15}.} This is done recursively, as follows:
  \begin{itemize}
  \item Start with $B_0:=B'$.
  \item Let $B_1$ be the $C^*$-subalgebra of $A$ generated by $B_0$ and 
    \begin{equation*}
      (\omega\otimes\id)\rho(B_0)(C_0(\bG)\otimes \bC)\subseteq A
    \end{equation*}
    as $\omega$ ranges over $C_0(\bG)^*$.
  \item Similarly, $B_2$ is the $C^*$-subalgebra of $A$ generated by $B_1$ and
    \begin{equation*}
      (\omega\otimes\id)\rho(B_1)(C_0(\bG)\otimes \bC)\subseteq A,\ \omega\in C_0(\bG)^*.
    \end{equation*}
  \item Continuing the process, we obtain a unital $C^*$-subalgebra
    \begin{equation*}
      B:=\overline{\bigcup_{i\ge 0} B_i}^{\|\cdot\|}
    \end{equation*}
    of $A$. 
  \end{itemize}
  It is clear by construction that we have
  \begin{equation*}
    \rho(B)(C_0(\bG)\otimes \bC)\subseteq C_0(\bG)\otimes B,
  \end{equation*}
  concluding Step (1).

  {\bf Step (2): enlarging a small $B'\subseteq A$ to a small $B\subseteq A$ satisfying both \Cref{item:15} and \Cref{item:16}.}

  Initially, $B_0:=B'$ might have the drawback that
  \begin{equation*}
    \overline{\rho(B_0)(C_0(\bG)\otimes \bC)}^{\|\cdot\|}\subsetneq C_0(\bG)\otimes B_0.
  \end{equation*}
  Nevertheless, because we do have \Cref{eq:aactcont}, $B_0$ can be enlarged to a (still small) $C^*$-algebra $B_0'$ with
  \begin{equation*}
    \overline{\rho(B_0')(C_0(\bG)\otimes \bC)}^{\|\cdot\|} \supseteq C_0(\bG)\otimes B_0.
  \end{equation*}
  Furthermore, Step (1) allows us to further extend $B_0'$ to $B_1$ satisfying both \Cref{item:15} and 
  \begin{equation*}
    \overline{\rho(B_1)(C_0(\bG)\otimes \bC)}^{\|\cdot\|} \supseteq C_0(\bG)\otimes B_0.
  \end{equation*}
  As before, continue the process recursively: $B_{n+1}$ is to $B_n$ as $B_1$ is to $B_0$. The resulting (small) $C^*$-algebra
  \begin{equation*}
    B:=\overline{\bigcup_{i\ge 0} B_i}^{\|\cdot\|}
  \end{equation*}
  will have both desired properties: \Cref{item:15} and \Cref{item:16}.
\end{proof}

\section{Limit preservation and (co)monads}\label{se:lim}

To reiterate and slightly amplify \Cref{re:clsdvr}, \Cref{th:ccmpl} and \Cref{cor:ladj} are, jointly, analogous to \cite[Theorem 3.2.4]{dvr-bk}: for any topological group $\bG$, the forgetful functor
\begin{equation}\label{eq:gtop}
  \cat{Top}^{\bG}\to \cat{Top}
\end{equation}
fro topological $\bG$-spaces to plain topological spaces has a left adjoint and creates colimits. It is natural to consider the dual problem of whether it creates/preserves {\it limits}, has a {\it right} adjoint, etc. As it happens, when $\bG$ is locally compact (which setup we are generalizing here), all of these questions have affirmative answers: \cite[Theorem 3.4.3]{dvr-bk}.

Dualizing from spaces to algebras, the analogue for us would be the question of whether the forgetful functor
\begin{equation}\label{eq:c1c1}
  \cC^{*\bG}_1\to \cC^*_1
\end{equation}
creates and/or preserves {\it limits}. There are several confounding factors making that discussion more involved than that of \Cref{eq:gtop}.

First, \cite[Proposition 3.4.2]{dvr-bk} says that \Cref{eq:gtop} {\it always} creates coproducts, whether the topological group $\bG$ is locally compact or not. In \Cref{eq:c1c1}, on the other hand, in working with {\it unital} $C^*$-algebras, we are intuitively dualizing the category $\cat{CH}$ of {\it compact Hausdorff} spaces. While in \cat{Top} coproducts are simply disjoint unions, in \cat{CH} they are formed by first taking the disjoint union and then applying the {\it Stone-\v{C}ech compactification}
\begin{equation*}
  \cat{Top}\ni \bX\mapsto \beta \bX\in \cat{CH},
\end{equation*}
i.e. the left adjoint of the inclusion functor $\cat{CH}\subset \cat{Top}$ (\cite[\S 38]{mnk} and \cite[\S V.8, following Corollary]{mcl}). In that setup, even classically (for compact $\bG$ and {\it commutative} unital $C^*$-algebras), the analogue of functor \Cref{eq:c1c1} cannot create coproducts.

\begin{example}\label{ex:notcoprod}
  Consider the additive group
  \begin{equation*}
    \bG:=\bZ_2=\varprojlim_n \bZ/2^n
  \end{equation*}
  of 2-adic integers, acting on each of its quotient groups $\bX_n:=\bZ/2^n$ by translation.

  Consider a continuous $\{0,1\}$-valued function $f$ on
  \begin{equation*}
    \beta\bY\quad\text{where}\quad \bY:=\coprod_n \bX_n
  \end{equation*}
  defined by extending the function taking the value $0$ at every trivial element $0_n\in \bX_n$ and 1 elsewhere on $\bY$. We will make some reference to {\it ultrafilters} on a set $\bS$ (i.e. the elements of $\beta\bS$), for which the reader will find a very pithy recollection in \cite[Introduction]{brgl}. For any such ultrafilter
  \begin{equation*}
    \cU\in \beta\bN
  \end{equation*}
  the {\it $\cU$-limit} \cite[\S 7]{brgl} 
  \begin{equation*}
    x:=\cU-\lim_n 0_n\in \beta\bY
  \end{equation*}
  belongs by construction to $f^{-1}(0)$. On the other hand, because every subgroup
  \begin{equation*}
    2^m\bZ_2\subset \bZ_2,\ m\in \bN
  \end{equation*}
  moves all $0_n$ for sufficiently large $n$, the action of any such subgroup on $x$ moves the latter into $f^{-1}(1)$.

  The conclusion is that no neighborhood of $0\in \bZ_2$, no matter how small, keeps $x$ in the neighborhood $f^{-1}(0)\ni x$. In short, the action
  \begin{equation*}
    \bZ_2\times \beta\bY\to \beta \bY
  \end{equation*}
  is not continuous, and hence \Cref{eq:c1c1} cannot create arbitrary (infinite) products.
\end{example}

In light of \Cref{ex:notcoprod}, it is only reasonable to inquire into {\it finite} limit preservation by \Cref{eq:c1c1}. With that caveat, products are unproblematic:

\begin{proposition}\label{pr:finprod}
  For any locally compact quantum group $\bG$ the forgetful functor $\cC^{*\bG}_1\to \cC^*_1$ creates finite products.
\end{proposition}

We leave the following simple preliminary remark to the reader.

\begin{lemma}\label{le:finprod}
  For any $C^*$-algebra $C\in \cC^*$ the endofunctors $C\otimes -$ and \Cref{eq:isfnc} of $\cC^*$ both preserve finite products.  \qedhere
\end{lemma}

\pf{pr:finprod}
\begin{pr:finprod}
  Let $\rho_i:A_i\circlearrowleft\bG$, $i\in I$ be a family of $\bG$-actions on unital $C^*$-algebras, and write
  \begin{equation*}
    \cC^*_1\ni A:=\prod_i A_i
  \end{equation*}
  for the product, equipped with the product-structure morphisms $\pi_i:A\to A_i$, $i\in I$. By \Cref{le:finprod} we have a canonical identification (the finiteness of $I$ is crucial!)
  \begin{equation*}
    M(C_0(\bG)\otimes A)
    =
    M\left(C_0(\bG)\otimes\prod A_i\right)
    \cong
    \prod M(C_0(\bG)\otimes A_i),
  \end{equation*}
  so there is a unique morphism $\rho:A\to M(C_0(\bG)\otimes A)$ making all squares
  \begin{equation*}
    \begin{tikzpicture}[auto,baseline=(current  bounding  box.center)]
      \path[anchor=base] 
      (0,0) node (l) {$A$}
      +(2,.5) node (u) {$M(C_0(\bG)\otimes A)$}
      +(4,-.5) node (d) {$A_i$}
      +(6,0) node (r) {$M(C_0(\bG)\otimes A_i)$}
      ;
      \draw[->] (l) to[bend left=6] node[pos=.5,auto] {$\scriptstyle \rho$} (u);
      \draw[->] (u) to[bend left=6] node[pos=.5,auto] {$\scriptstyle \id\otimes\pi_i$} (r);
      \draw[->] (l) to[bend right=6] node[pos=.5,auto,swap] {$\scriptstyle \pi_i$} (d);
      \draw[->] (d) to[bend right=6] node[pos=.5,auto,swap] {$\scriptstyle \rho_i$} (r);
    \end{tikzpicture}
  \end{equation*}
  commute. Its coassociativity follows from the fact that
  \begin{equation*}
    (\id\otimes\rho)\rho\quad\text{and}\quad (\Delta\otimes\id)\rho:A\to M(C_0(\bG)^{\otimes 2}\otimes A)
  \end{equation*}
  both fit as the unique upper left-hand map rendering all 
  \begin{equation*}
    \begin{tikzpicture}[auto,baseline=(current  bounding  box.center)]
      \path[anchor=base] 
      (0,0) node (l) {$A$}
      +(3,.5) node (u) {$M(C_0(\bG)^{\otimes 2}\otimes A)$}
      +(4,-.5) node (d) {$A_i$}
      +(7,0) node (r) {$M(C_0(\bG)^{\otimes 2}\otimes A_i)$}
      ;
      \draw[->] (l) to[bend left=6] node[pos=.5,auto] {$\scriptstyle $} (u);
      \draw[->] (u) to[bend left=6] node[pos=.5,auto] {$\scriptstyle \id\otimes\pi_i$} (r);
      \draw[->] (l) to[bend right=6] node[pos=.5,auto,swap] {$\scriptstyle \pi_i$} (d);
      \draw[->] (d) to[bend right=6] node[pos=.5,auto,swap] {$\scriptstyle (\id\otimes\rho_i)\rho_i = (\Delta\otimes\id)\rho_i$} (r);
    \end{tikzpicture}
  \end{equation*}
  commutative, while the continuity condition \Cref{eq:aactcont} for $\rho$ follows from the analogous conditions for its individual $A_i$-components (the finiteness of $I$ once more being essential).
\end{pr:finprod}

In the presence of finite products (taken care of by \Cref{pr:finprod}), finite-limit creation/preservation amounts to creating, or, respectively preserving, either pullbacks or, equivalently, equalizers \cite[Theorem 12.4 and Proposition 13.3]{ahs}.

Changing perspective slightly, recall that the issue of whether or not monomorphisms in $\cC^{*\bG}_1$ are injective was left standing in \Cref{se:epimono}; as \Cref{le:monchar} makes clear, in order to make that assessment one would need to understand pullbacks in $\cC^{*\bG}_1$. Pullbacks (or, what amounts essentially to the same thing, equalizers) are also what mandates local compactness in the aforementioned \cite[Theorem 3.4.3]{dvr-bk}: while \Cref{eq:gtop} always creates coproducts \cite[Theorem 3.4.2]{dvr-bk}, it does not, in general, preserve pushouts \cite[\S 3.4.4 and Theorem 3.4.5]{dvr-bk}.

Here, the nature problem is slightly different: we are already working with {\it locally compact} quantum groups, so one should presumably expect equalizer creation/preservation for $\cC^{*\bG}_1\to \cC^*_1$ (recall that this forgetful functor is to be thought of as a bi-opposite analogue of \Cref{eq:gtop}, hence the passage from colimits to limits, etc.). The {\it quantum} setup, though, has its own peculiar difficulties:
\begin{enumerate}[(a)]
\item\label{item:22} the tensor-product functor $C_0(\bG)\otimes -$ does not play well with equalizers;
\item\label{item:23} and in general, the structure map
  \begin{equation*}
    \rho:A\to M(C_0(\bG)\otimes A)
  \end{equation*}
  need not be injective (this issue occasionally comes up in the types of arguments we carry out below).

  In relation to this latter observation, it is not uncommon for some authors to either require injectivity for the classes of actions of interest (what \cite[D\'efinition 0.2]{bs-cross} or \cite[\S 3.1]{vrgn-phd} would call {\it $C_0(\bG)$-algebras}, for instance, require injectivity), or to at least specialize to injective actions in specific results (e.g. \cite[\S 6]{mrw}).
\end{enumerate}

There is a convenient way to address both issues simultaneously, by specializing to a class of LCQGs still wider than that of classical locally compact groups (\cite[Definition 3.1]{bt}):

\begin{definition}\label{def:coam}
  A locally compact quantum group $\bG$ is {\it coamenable} if $C_0(\bG)$ has a {\it counit}: a $C^*$-algebra morphism
  \begin{equation*}
    \varepsilon=\varepsilon_{\bG}:C_0(\bG)\to \bC
  \end{equation*}
  such that
  \begin{equation}\label{eq:coun}
    (\id\otimes\varepsilon)\Delta_{\bG}
    = \id =
    (\varepsilon\otimes\id)\Delta_{\bG} = \id.
  \end{equation}

  As explained in the discussion following \cite[Definition 3.1]{bt}, it is enough to require one of the equalities in \Cref{eq:coun}; the other follows. 
\end{definition}

For coamenable LCQGs the defining property \Cref{eq:coun} of the counit extends to actions, familiar from the theory of {\it comodule algebras} over bialgebras \cite[\S 11.3]{rad}. We will also have to assume henceforth that our locally compact quantum groups are {\it regular} (\cite[D\'efinition 3.3]{bs-cross}, \cite[Definition 2.8]{vs-impr}). We refer to \cite[\S 2.1]{vs-impr} (and its sources) for the background necessary for the definition below; the notion will play here only a technical, black-box role.

\begin{definition}
  Let $\bG$ be an LCQG and denote by
  \begin{itemize}
  \item $L^2=L^2(\bG)$ the carrier Hilbert space of the GNS representation attached to the left Haar weight;
  \item and
    \begin{equation*}
      W=W_{\bG} \in B(L^2\otimes L^2)
    \end{equation*}
    the {\it multiplicative unitary} implementing the comultiplication of $\bG$ by
    \begin{equation*}
      C_0(\bG)\ni x\mapsto \Delta_{\bG}(x) = W^*(1\otimes x)W \in M(C_0(\bG)\otimes C_0(\bG))\subset B(L^2\otimes L^2).
    \end{equation*}
  \end{itemize}
  $\bG$ is {\it regular} if the norm-closed span of
  \begin{equation*}
    \{(\id\otimes\omega)(\text{flip}(W))\ |\ \omega\in K(L^2)^*\}
  \end{equation*}
  is precisely the $C^*$-algebra $K(L^2)$ of compact operators on $L^2$.
\end{definition}

\begin{proposition}\label{pr:coun}
  Let $\bG$ be a coamenable LCQG and $\varepsilon=\varepsilon_{\bG}$ its counit.
  \begin{enumerate}[(a)]
  \item\label{item:26} For any action $\rho:A\circlearrowleft \bG$ on a $C^*$-algebra the map
    \begin{equation*}
      (\varepsilon\otimes \id)\rho:M(A)\to M(A)
    \end{equation*}
    is the identity. 
  \item\label{item:27} If $\bG$ is in addition regular, then conversely, a non-degenerate $\rho:A\to M(C_0(\bG)\otimes A)$ satisfying conditions \Cref{item:6} and \Cref{item:7} of \Cref{def:act} is an action provided
    \begin{equation}\label{eq:eid}
      (\varepsilon\otimes \id)\rho = \id_{M(A)}. 
    \end{equation}
  \end{enumerate}
\end{proposition}
\begin{proof}
  \begin{enumerate}[(a)]
  \item Denote
    \begin{equation*}
      r:=(\varepsilon\otimes \id)\rho : M(A)\to M(A).
    \end{equation*}
    This is a strictly-continuous unital $C^*$ morphism, and an application of $\varepsilon\otimes \id$ to the two sides of \Cref{eq:aactcont} shows that $r|_A$ has norm-dense image in $A$, so that $r$ has strictly-dense image. Next, composing
    \begin{equation*}
      (\id\otimes\rho)\rho = (\Delta_{\bG}\otimes\id)\rho
      : M(A)\to M(C_0(\bG)^{\otimes 2}\otimes A)
    \end{equation*}
    with $\varepsilon\otimes\varepsilon\otimes\id$ also shows (via \Cref{eq:coun}) that $r:M(A)\to M(A)$ is idempotent. But now we are done: being idempotent $r$ is the identity on a strictly-dense subset
    \begin{equation*}
      r(M(A))\subseteq M(A),
    \end{equation*}
    so it must be the identity everywhere by strict continuity.
  \item The assumption \Cref{eq:eid} implies the {\it weak continuity} condition of \cite[discussion preceding Proposition 5.8]{bsv}, so under the regularity assumption the conclusion follows from \cite[Proposition 5.8]{bsv}.
  \end{enumerate}
  
  This finishes the proof of the two claims. 
\end{proof}

\begin{remark}
  In \Cref{pr:coun} \Cref{item:27} regularity matters: the {\it quantum $E(2)$ group} studied in \cite{jac-phd} is
  \begin{itemize}
  \item coamenable \cite[Theorem 3.2.29]{jac-phd};    
  \item and {\it semi-regular} in the sense of \cite[Definition 1.3.3]{jac-phd} by \cite[Corollary 2.8.25]{jac-phd};
  \item so the example of \cite[Proposition 5.8]{bsv} applies. 
  \end{itemize}
  That example is of a ``not-quite-action'' that is easily checked to satisfy \Cref{eq:eid} but {\it not} \Cref{def:act} \Cref{item:8} (hence not an action).
\end{remark}


\Cref{pr:coun} will be particularly useful when attempting to restrict actions to $C^*$-subalgebras, ensuring that we again obtain actions.

\begin{lemma}\label{le:resact}
  Let $\bG$ be a regular coamenable LCQG and $\rho:A\circlearrowleft\bG$ an action. If $A'\le A$ is a $C^*$-subalgebra for which
  \begin{equation*}
    \rho(A')\subseteq M_{A'}(C_0(\bG)\otimes A')\le M_{A}(C_0(\bG)\otimes A)
  \end{equation*}
  and $\rho|_{A'}$ is non-degenerate as a map to $M(C_0(\bG)\otimes A')$, then the restriction $\rho':=\rho|_{A'}$ is a $\bG$-action.
\end{lemma}
\begin{proof}

  
  By \Cref{pr:coun} \Cref{item:26} $\rho$ satisfies \Cref{eq:eid}, and hence so does its restriction $\rho'$ to $A'$. But then \Cref{pr:coun} \Cref{item:27} applies, and we are done.
\end{proof}

We can now revisit (and resolve) the two bothersome items \Cref{item:22} and \Cref{item:23} above:

\begin{enumerate}[(a$^*$)]
\item\label{item:24} When $\bG$ is coamenable its function algebra $C_0(\bG)$ 
  \begin{itemize}
  \item is {\it nuclear} by \cite[Theorem 3.3]{bt}, in the sense of \cite[\S IV.3.1]{blk}: full and reduced $C^*$ tensor products by $C_0(\bG)$ coincide canonically;
  \item hence also {\it exact} \cite[\S II.9.6.6]{blk}: $C_0(\bG)\otimes -$ preserves short exact sequences;
  \item so that $C_0(\bG)\otimes-$ preserves pullbacks \cite[Theorem 3.9]{ped-pp}.
  \end{itemize}
\item\label{item:25} Any action $\rho:A\to M(C_0(\bG)\otimes A)$ is indeed injective, by \Cref{pr:coun} \Cref{item:26}.
\end{enumerate}

\begin{theorem}\label{th:finlim}
  For a regular coamenable LCQG $\bG$ the forgetful functor $\cC^{*\bG}_1\to \cC^*_1$ creates finite limits.  
\end{theorem}

We first need the arbitrary-finite-limit version of \Cref{le:finprod}.

\begin{lemma}\label{le:finlim}
  For any exact $C^*$-algebra $C\in \cC^*$ the endofunctors $C\otimes -$ and \Cref{eq:isfnc} of $\cC^*$ both preserve finite limits.
\end{lemma}
\begin{proof}
  Having disposed of finite products (\Cref{le:finprod}), it will be enough to handle either pullbacks or equalizers \cite[Proposition 13.3]{ahs} (we will switch between the two freely). For $C\otimes -$ this is precisely what \cite[Theorem 3.9]{ped-pp} does, so we are left having to show that assuming $C$ is exact,
  \begin{equation*}
    \cC^*\ni A\mapsto M_A(C\otimes A) \in \cC^*
  \end{equation*}
  preserves equalizers. To that end, consider a parallel pair $\iota_i:B\to B'$, $i=1,2$ of morphisms in $\cC^*$ together with their equalizer $\iota:A\to B$. That $M_A(C\otimes A)$ maps (injectively: \Cref{le:isfnc}) to
  \begin{equation*}
    X:=\text{equalizer of }\id\otimes \iota_i:M_B(C\otimes B)\to M_{B'}(C\otimes B'),\ i=1,2
  \end{equation*}
  follows from the universality property of that equalizer; it thus remains to argue that that canonical map is onto or, equivalently, that every element
  \begin{equation*}
    x\in X\subseteq M_B(C\otimes B)
  \end{equation*}
  in fact belongs to $M_A(C\otimes A)$. This, by definition, means that $x(c\otimes 1)\in C\otimes A$ for arbitrary $c\in C$. Since $x(c\otimes 1)$ certainly belongs to the equalizer of
  \begin{equation*}
    \id\otimes \iota_i:C\otimes B\to C\otimes B',
  \end{equation*}
  the conclusion follows from the claim on $C\otimes -$.
\end{proof}

\pf{th:finlim}
\begin{th:finlim}
  
  
  Once more, an application of \cite[Proposition 13.3]{ahs} and the fact that we already have the claim for finite {\it products} reduces the discussion to equalizers. We thus fix actions
  \begin{equation*}
    \rho_B:B\circlearrowleft\bG
    \quad\text{and}\quad
    \rho_{B'}:B'\circlearrowleft\bG
  \end{equation*}
  in $\cC^{*\bG}_1$ and equivariant morphisms $\iota_i:B\to B'$, $i=1,2$, as well as their equalizer $\iota:A\to B$ in $\cC^*_1$.

  Write $C:=C_0(\bG)$. First, assuming an action $\rho:=\rho_A:A\circlearrowleft\bG$ making $\iota:A\to B$ equivariant, that equivariance condition
  \begin{equation}\label{eq:wantequiv}
    \begin{tikzpicture}[auto,baseline=(current  bounding  box.center)]
      \path[anchor=base] 
      (0,0) node (l) {$A$}
      +(2,.5) node (u) {$M(C\otimes A)$}
      +(3,-.5) node (d) {$B$}
      +(5,0) node (r) {$M(C\otimes B)$}
      ;
      \draw[->] (l) to[bend left=6] node[pos=.5,auto] {$\scriptstyle \rho$} (u);
      \draw[right hook->] (u) to[bend left=6] node[pos=.5,auto] {$\scriptstyle \id\otimes\iota$} (r);
      \draw[right hook->] (l) to[bend right=6] node[pos=.5,auto,swap] {$\scriptstyle \iota$} (d);
      \draw[->] (d) to[bend right=6] node[pos=.5,auto,swap] {$\scriptstyle \rho_B$} (r);
    \end{tikzpicture}
  \end{equation}
  will determine $\rho$ uniquely by the injectivity of the upper right-hand arrow map (\Cref{le:isfnc}). The {\it creation} part of the claim thus follows, assuming we actually construct such an action $\rho_A$.

  Furthermore, the selfsame desired constraint \Cref{eq:wantequiv} shows that $\rho$ must be the restriction of $\rho_B$ to $A$. On the other hand, assuming $\rho_B$ does indeed map $A$ to $M(C\otimes A)$ it will be an action by \Cref{le:resact}. In conclusion, all we need is
  \begin{equation*}
    \rho_B(A)\subseteq M(C\otimes A)\le M(C\otimes B). 
  \end{equation*}
  To see that this is indeed the case note that
  \begin{itemize}
  \item the upper right-hand arrow in \Cref{eq:wantequiv} is the equalizer of
    \begin{equation}\label{eq:idii}
      \id\otimes\iota_i:M(C\otimes B)\to M(C\otimes B')
    \end{equation}
    (\Cref{le:finlim});
  \item and the lower composition $\rho_B\iota$ equalizes the two morphisms \Cref{eq:idii}.
  \end{itemize}
  As remarked, this finishes the proof. 
\end{th:finlim}

As noted before, \Cref{th:finlim} has a bearing on the description of monomorphisms in the category $\cC^{*\bG}_1$.

\begin{corollary}\label{cor:moninj}
  For any regular coamenable locally compact group $\bG$ the monomorphisms in $\cC^{*\bG}_1$ are precisely the injective morphisms.
\end{corollary}
\begin{proof}
  As in the discussion preceding \Cref{th:episurj}, that injectivity implies the property of being monic is immediate. The converse follows from the criterion in \Cref{le:monchar} together with the fact that by \Cref{th:finlim} pullbacks in $\cC^{*\bG}_1$ are computed as in $\cC^*_1$ (and hence as in $\cat{Set}$). 
\end{proof}

\subsection{The action comonad}\label{subse:cmnd}

Classically, \Cref{eq:gtop} is a good deal more than a limit-creating right adjoint: according to the proof of \cite[Theorem 3.2.4]{dvr-bk} it is {\it monadic}. We recall some of the language; some of the numerous sources include \cite[Chapter VI]{mcl}, \cite[Chapter 3]{bw}, \cite[\S V.20]{ahs} and \cite[\S 0.4]{dvr-bk}.

\begin{definition}\label{def:monad}
  Let $\cC$ be a category.
  \begin{enumerate}[(a)]
  \item A {\it monad} or {\it triple} on $\cC$ is an endofunctor $T:\cC\to \cC$ equipped with natural transformations
    \begin{equation*}
      \eta:\id\to T
      \quad\text{and}\quad
      \mu:T^2\to T
    \end{equation*}
    making $T$ into a {\it unital algebra} in the monoidal category of $\cC$-endofunctors. In other words, we require associativity
    \begin{equation*}
      \begin{tikzpicture}[auto,baseline=(current  bounding  box.center)]
        \path[anchor=base] 
        (0,0) node (l) {$T^3$}
        +(2,.5) node (u) {$T^2$}
        +(2,-.5) node (d) {$T^2$}
        +(4,0) node (r) {$T$}
        ;
        \draw[->] (l) to[bend left=6] node[pos=.5,auto] {$\scriptstyle \mu T$} (u);
        \draw[->] (u) to[bend left=6] node[pos=.5,auto] {$\scriptstyle \mu$} (r);
        \draw[->] (l) to[bend right=6] node[pos=.5,auto,swap] {$\scriptstyle T\mu$} (d);
        \draw[->] (d) to[bend right=6] node[pos=.5,auto,swap] {$\scriptstyle \mu$} (r);
      \end{tikzpicture}
    \end{equation*}
    and unitality:
    \begin{equation*}
      \begin{tikzpicture}[auto,baseline=(current  bounding  box.center)]
        \path[anchor=base] 
        (0,0) node (l) {$T^2$}
        +(2,0.8) node (u) {$T$}
        +(2,-.8) node (d) {$T$}
        +(4,0) node (r) {$T$.}
        ;
        \draw[<-] (l) to[bend left=6] node[pos=.5,auto] {$\scriptstyle T\eta$} (u);
        \draw[->] (u) to[bend left=6] node[pos=.5,auto] {$\scriptstyle \id$} (r);
        \draw[<-] (l) to[bend right=6] node[pos=.5,auto,swap] {$\scriptstyle \eta T$} (d);
        \draw[->] (d) to[bend right=6] node[pos=.5,auto,swap] {$\scriptstyle \id$} (r);
        \draw[->] (l) to[bend left=0] node[pos=.5,auto] {$\scriptstyle \mu$} (r);
      \end{tikzpicture}
    \end{equation*}
  \item For a monad $T$ a {\it $T$-algebra} is a pair $(c,s)$ consisting of an object $c\in \cC$ and a morphism $s:Tc\to c$ in $\cC$ satisfying the conditions
    \begin{equation*}
      \begin{tikzpicture}[auto,baseline=(current  bounding  box.center)]
        \path[anchor=base] 
        (0,0) node (l) {$T^2c$}
        +(2,.5) node (u) {$Tc$}
        +(2,-.5) node (d) {$Tc$}
        +(4,0) node (r) {$c$}
        ;
        \draw[->] (l) to[bend left=6] node[pos=.5,auto] {$\scriptstyle \mu_c$} (u);
        \draw[->] (u) to[bend left=6] node[pos=.5,auto] {$\scriptstyle s$} (r);
        \draw[->] (l) to[bend right=6] node[pos=.5,auto,swap] {$\scriptstyle Ts$} (d);
        \draw[->] (d) to[bend right=6] node[pos=.5,auto,swap] {$\scriptstyle s$} (r);
      \end{tikzpicture}
    \end{equation*}
    and
    \begin{equation*}
      \begin{tikzpicture}[auto,baseline=(current  bounding  box.center)]
        \path[anchor=base] 
        (0,0) node (l) {$c$}
        +(2,.5) node (u) {$Tc$}
        +(4,0) node (r) {$c$.}
        ;
        \draw[->] (l) to[bend left=6] node[pos=.5,auto] {$\scriptstyle \mu_c$} (u);
        \draw[->] (u) to[bend left=6] node[pos=.5,auto] {$\scriptstyle s$} (r);
        \draw[->] (l) to[bend right=6] node[pos=.5,auto,swap] {$\scriptstyle \id$} (r);
      \end{tikzpicture}
    \end{equation*}
  \item The $T$-algebras form a category $\cC^T$ (the {\it Eilenberg-Moore category} of $T$), having as morphisms $(c,s)\to (c',s')$ those maps $f:c\to c'$ in $\cC$ that make
    \begin{equation*}
      \begin{tikzpicture}[auto,baseline=(current  bounding  box.center)]
        \path[anchor=base] 
        (0,0) node (l) {$Tc$}
        +(2,.5) node (u) {$Tc'$}
        +(2,-.5) node (d) {$c$}
        +(4,0) node (r) {$c'$}
        ;
        \draw[->] (l) to[bend left=6] node[pos=.5,auto] {$\scriptstyle Tf$} (u);
        \draw[->] (u) to[bend left=6] node[pos=.5,auto] {$\scriptstyle s'$} (r);
        \draw[->] (l) to[bend right=6] node[pos=.5,auto,swap] {$\scriptstyle s$} (d);
        \draw[->] (d) to[bend right=6] node[pos=.5,auto,swap] {$\scriptstyle f$} (r);
      \end{tikzpicture}
    \end{equation*}
    commute.
  \item A functor $\cD\to \cC$ is {\it monadic} or {\it tripleable} if it is naturally isomorphic to the forgetful functor $\cC^T\to \cC$ for a monad $T$ on $\cC$.
  \end{enumerate}
  We have dual notions: comonads $S$ on $\cC$, corresponding categories $\cC_S$ of {\it $S$-coalgebras}, comonadic functors, and so on, which we leave it to the reader to unpack. One can dispatch them all formally by simply noting that a comonad on $\cC$ is a monad on the opposite category $\cC^{op}$.
\end{definition}

As explained in \cite[discussion preceding Definition 20.1]{ahs} and further illustrated in \cite[\S 24]{}, monadic functors are meant to capture the intuition of ``forgetting some algebraic structure''. To wit, examples of monadic functors include (by \cite[\S VI.8, Theorem 1]{mcl}), for instance, the forgetful functors from
\begin{itemize}
\item groups;
\item or monoids;
\item or rings;
\item or Lie algebras;
\end{itemize}
etc. etc. to \cat{Set}. There is also the motivating (for us, here) example of \Cref{eq:gtop}, which fits into the same framework of forgetting algebraic structure (namely the structure consisting of the operators on a space induced by a $\bG$-action). The following result recasts that monadicity result in its appropriate non-commutative context, keeping in mind the requisite dualization in passing from spaces to $C^*$-algebras. 

\begin{theorem}\label{th:coamcomon}
  For any regular coamenable locally compact quantum group $\bG$ the forgetful functor $\cC^{*\bG}_1\to \cC^*_1$ is comonadic.
\end{theorem}
\begin{proof}
  This will be a straightforward application of {\it Beck's monadicity theorem} (\cite[\S 3.3, Theorem 10]{bw}). We have to check that 
  \begin{itemize}
  \item the forgetful functor is a left adjoint, which we know from \Cref{cor:ladj};
  \item and reflects isomorphisms, which is obvious: in either category isomorphisms are bijective morphisms;
  \item and that $\cC^{*\bG}_1$ has certain equalizers which the forgetful functor then preserves.

    We are not concerned here with the precise class of equalizers one is usually interested in: we already know they all exist (\Cref{cor:allgood}) and (for coamenable $\bG$) are all preserved (\Cref{th:finlim}).
  \end{itemize}
  This finishes the proof.
\end{proof}

\addcontentsline{toc}{section}{References}

\Addresses

\end{document}